\documentclass[11pt,reqno]{amsart}
\usepackage[hypertex]{hyperref}

\usepackage{amsmath,amsthm}
\usepackage{amscd}
\usepackage{amssymb}
\usepackage{euscript}
\usepackage[frame,ps,matrix,arrow,curve,rotate]{xy}
\usepackage{epic,eepic}

\numberwithin{equation}{subsection}

\newcommand{\sqsp}{\renewcommand{\baselinestretch}{1.3}\tiny\normalsize}

\newcommand{\relphantom[1]}{\mathrel{\phantom{#1}}}

\newtheorem{theorem}[subsection]{Theorem}

\newtheorem{proposition}[subsection]{Proposition}
\newtheorem{corollary}[subsection]{Corollary}

\theoremstyle{definition}
\newtheorem{definition}[subsection]{Definition}
\newtheorem{example}[subsection]{Example}

\newtheorem{convention}[subsection]{Convention}

\newcommand{\bk}{\mathbf{k}}

\newcommand{\frakC}{\mathfrak{C}}
\newcommand{\sfP}{\mathsf{P}}
\newcommand{\sfF}{\mathsf{F}}

\newcommand{\UGr}{\mathsf{UGr}}

\DeclareMathOperator{\As}{\mathbf{As}}
\DeclareMathOperator{\AsMod}{\mathbf{AsMod}}
\DeclareMathOperator{\Lie}{\mathbf{Lie}}
\DeclareMathOperator{\ModAlg}{\mathbf{ModAlg}}
\DeclareMathOperator{\Ent}{\mathbf{Ent}}

\DeclareMathOperator{\YD}{\mathbf{YD}}
\DeclareMathOperator{\HopfMod}{\mathbf{HopfMod}}
\DeclareMathOperator{\Id}{Id} \DeclareMathOperator{\Der}{Der}
 \DeclareMathOperator{\Hom}{Hom}
\DeclareMathOperator{\End}{End} \DeclareMathOperator{\colim}{colim}

\def\calC{{\EuScript C}} \def\calD{{\EuScript D}}
\def\calF{{\EuScript F}} \def\seq#1#2{\{#1\}_{{#2}}}
\def\widedef{{\hskip 4mm := \hskip 4mm}}
\def\bfD{{\mathbf D}}
 \def\Ob{\mathrm {Ob}}
\def\Iso{{\mathbf {Iso}}}
\def\Riso{{\mathcal R}_{\rm iso}}
\def\pa{\partial}
\def\CIso#1{C^{#1}_\Iso(T,T)}
\def\Rada#1#2#3{#1_{#2},\dots,#1_{#3}} 

\def\ZZbbZbZ{
\thicklines
{
\unitlength=.3pt
\begin{picture}(80.00,40.00)(0.00,40.00)
\put(40.00,40.00){\makebox(0.00,0.00){$\bullet$}}
\put(20.00,20.00){\makebox(0.00,0.00){$\bullet$}}
\put(40.00,0.00){\line(-1,1){20.00}}
\put(40.00,40.00){\line(0,1){40.00}}
\put(80.00,0.00){\line(-1,1){40.00}}
\put(0.00,0.00){\line(1,1){40.00}}
\end{picture}}
\raisebox{-02pt}{\rule{0pt}{2pt}}
}

\def\CCbbCbC{
\thicklines
{
\unitlength=.3pt
\begin{picture}(80.00,40.00)(0.00,-40.00)
\put(40.00,-40.00){\makebox(0.00,0.00){$\bullet$}}
\put(20.00,-20.00){\makebox(0.00,0.00){$\bullet$}}
\put(40.00,0.00){\line(-1,-1){20.00}}
\put(40.00,-40.00){\line(0,-1){40.00}}
\put(80.00,0.00){\line(-1,-1){40.00}}
\put(0.00,0.00){\line(1,-1){40.00}}
\end{picture}}
\raisebox{-02pt}{\rule{0pt}{2pt}}
}

\def\ZbZbbZZ{
\thicklines
{
\unitlength=.3pt
\begin{picture}(80.00,40.00)(0.00,40.00)
\put(40.00,40.00){\makebox(0.00,0.00){$\bullet$}}
\put(60.00,20.00){\makebox(0.00,0.00){$\bullet$}}
\put(40.00,0.00){\line(1,1){20.00}}
\put(40.00,40.00){\line(0,1){40.00}}
\put(80.00,0.00){\line(-1,1){40.00}}
\put(0.00,0.00){\line(1,1){40.00}}
\end{picture}}
\raisebox{-20pt}{\rule{0pt}{2pt}}
}

\def\CbCbbCC{
\thicklines
{
\unitlength=.3pt
\begin{picture}(80.00,40.00)(0.00,-40.00)
\put(40.00,-40.00){\makebox(0.00,0.00){$\bullet$}}
\put(60.00,-20.00){\makebox(0.00,0.00){$\bullet$}}
\put(40.00,0.00){\line(1,-1){20.00}}
\put(40.00,-40.00){\line(0,-1){40.00}}
\put(80.00,0.00){\line(-1,-1){40.00}}
\put(0.00,0.00){\line(1,-1){40.00}}
\end{picture}}
\raisebox{-20pt}{\rule{0pt}{2pt}}
}

\def\BfZZbbZbZ{
\thicklines
{
\unitlength=.3pt
\begin{picture}(80.00,40.00)(0.00,40.00)
\put(40.00,40.00){\makebox(0.00,0.00){$\bullet$}}
\put(20.00,20.00){\makebox(0.00,0.00){$\bullet$}}
\put(40.00,0.00){\line(-1,1){20.00}}
\put(37.50,40.00){\line(0,1){40.00}}
\put(42.50,40.00){\line(0,1){40.00}}
\put(80.00,3.50){\line(-1,1){40.00}}
\put(76.50,0.00){\line(-1,1){40.00}}
\put(0.00,0.00){\line(1,1){40.00}}
\end{picture}}
\raisebox{-20pt}{\rule{0pt}{2pt}}
}

\def\BfZbZbbZZ{
\thicklines
{
\unitlength=.3pt
\begin{picture}(80.00,40.00)(0.00,40.00)
\put(40.00,40.00){\makebox(0.00,0.00){$\bullet$}}
\put(60.00,20.00){\makebox(0.00,0.00){$\bullet$}}
\put(40.00,0.00){\line(1,1){20.00}}
\put(37.50,40.00){\line(0,1){40.00}}
\put(42.50,40.00){\line(0,1){40.00}}
\put(80.00,3.50){\line(-1,1){40.00}}
\put(76.50,0.00){\line(-1,1){40.00}}
\put(0.00,0.00){\line(1,1){40.00}}
\end{picture}}
\raisebox{-020pt}{\rule{0pt}{2pt}}
}

\def\BfZbZbbZZa{
\thicklines
{
\unitlength=.3pt
\begin{picture}(80.00,40.00)(0.00,40.00)
\put(40.00,40.00){\makebox(0.00,0.00){$\bullet$}}
\put(60.00,20.00){\makebox(0.00,0.00){$\bullet$}}
\put(42.50,1.00){\line(1,1){20.00}}
\put(38,4.50){\line(1,1){20.00}}
\put(37.50,40.00){\line(0,1){40.00}}
\put(42.50,40.00){\line(0,1){40.00}}
\put(80.00,3.50){\line(-1,1){40.00}}
\put(76.50,0.00){\line(-1,1){40.00}}
\put(0.00,0.00){\line(1,1){40.00}}
\end{picture}}
\raisebox{-020pt}{\rule{0pt}{2pt}}
}

\def\GerdaI{
{
\unitlength=.25pt
\begin{picture}(80.00,50.00)(0.00,50.00)
\put(15.00,20.00){\makebox(0.00,0.00){$\bullet$}}
\put(75.00,60.00){\makebox(0.00,0.00){$\bullet$}}
\put(15.00,60.00){\makebox(0.00,0.00){$\bullet$}}
\put(45.00,80.00){\makebox(0.00,0.00){$\bullet$}}
\put(15.00,0.00){\line(0,1){20.00}}
\put(3,0){\qbezier(76.00,62.00)(105.50,40.00)(105.00,0.00)}
\put(-3,0){\qbezier(74.00,60.00)(105.50,40.00)(105.00,0.00)}
\put(17.50,62.50){\line(1,-1){18}}
\put(12.50,57.50){\line(1,-1){18}}
\put(.5,2.5){\qbezier(50.00,30.00)(60.00,20.00)(63.00,-4.5)}
\put(-4.5,-2.5){\qbezier(50.00,30.00)(60.00,20.00)(60.00,0.00)}

\qbezier(15.00,60.00)(-15.00,40.00)(15.00,20.00)
\put(75.00,60.00){\line(-3,-2){60.00}}
\put(-3.5,0){\put(45.00,80.00){\line(0,1){20.00}}}
\put(3.5,0){\put(45.00,80.00){\line(0,1){20.00}}}
\put(2.2,2.2){\put(75.00,60.00){\line(-3,2){30.00}}}
\put(-2.2,-2.2){\put(75.00,60.00){\line(-3,2){30.00}}}
\put(15.00,60.00){\line(3,2){30.00}}
\end{picture}}
}

\def\GerdaII
{{
\unitlength=.33pt
\thicklines
\begin{picture}(80.00,40.00)(0.00,37.00)
\put(36.00,28.00){\makebox(0.00,0.00){$\bullet$}}
\qbezier(40.00,24.00)(0.00,72.00)(0.00,84.00)
\put(40.00,24.00){\line(1,-1){24.00}}
\put(4,0){
\put(52.00,60.00){\makebox(0.00,0.00){$\bullet$}}
\put(50.00,62.00){\line(-2,-3){40.00}}
\put(54.00,60.00){\line(-2,-3){40.00}}
\put(53.00,59.00){\line(1,1){24.00}}
\put(50.00,62.00){\line(1,1){24.00}}
\put(26.00,83.00){\line(1,-1){24.00}}
\put(29.00,86.00){\line(1,-1){24.00}}
}
\end{picture}}
}

\def\GerdaIIinv
{{
\unitlength=.33pt
\thicklines
\begin{picture}(80.00,40.00)(0.00,-42.00)
\put(36.00,-28.00){\makebox(0.00,0.00){$\bullet$}}
\put(3,0){\qbezier(37.00,-26.00)(0.00,-72.00)(0.00,-84.00)}
\put(-3,0){\qbezier(38.00,-24.00)(0.00,-72.00)(0.00,-84.00)}
\put(37.00,-22.00){\line(1,1){24.00}}
\put(40.00,-27.00){\line(1,1){25.00}}
\put(4,0){
\put(52.00,-60.00){\makebox(0.00,0.00){$\bullet$}}
\put(52.00,-60.00){\line(-2,3){40.00}}
\put(52.00,-60.00){\line(1,-1){24.00}}
\put(27.00,-84.00){\line(1,1){24.00}}
}
\end{picture}}
}

\def\GerdaIII
{{
\unitlength=.35pt
\thicklines
\begin{picture}(80.00,44.00)(0.00,35.00)
\put(40.00,40.00){\makebox(0.00,0.00){$\bullet$}}
\put(20.00,20.00){\makebox(0.00,0.00){$\bullet$}}
\put(20.00,60.00){\makebox(0.00,0.00){$\bullet$}}
\put(-2,0){\qbezier(18.00,60.00)(0.00,40.00)(20.00,20.00)}
\put(2,0){\qbezier(20.00,60.00)(0.00,40.00)(22.00,20.00)}
\put(38.00,82.00){\line(-1,-1){20.00}}
\put(41.00,79.00){\line(-1,-1){20.00}}
\put(0.00,80.00){\line(1,-1){40.00}}
\qbezier(40.00,40.00)(60.00,20.00)(60.00,0.00)
\put(21,19){\line(1,1){60.00}}
\put(18,22.0){\line(1,1){60.00}}
\put(18,0.00){\line(0,1){20.00}}
\put(22,0.00){\line(0,1){20.00}}
\end{picture}}
}

\def\GerdaIIIinv
{{
\unitlength=.35pt
\thicklines
\begin{picture}(80.00,44.00)(0.00,-42.00)
\put(40.00,-40.00){\makebox(0.00,0.00){$\bullet$}}
\put(20.00,-20.00){\makebox(0.00,0.00){$\bullet$}}
\put(20.00,-60.00){\makebox(0.00,0.00){$\bullet$}}
\put(40.00,-40.00){\line(1,-1){40.00}}
\qbezier(20.00,-60.00)(0.00,-40.00)(20.00,-20.00)
\put(40.00,-80.00){\line(-1,1){20.00}}
\put(-1.5,-78.50){\line(1,1){40.00}}
\put(1.50,-81.50){\line(1,1){40.00}}
\put(-3,0){\qbezier(40.00,-40.00)(60.00,-20.00)(60.00,0.00)}
\put(3,0){\qbezier(40.00,-40.00)(60.00,-20.00)(58.00,0.00)}
\put(20.00,-20.00){\line(1,-1){20.00}}
\put(20.00,0.00){\line(0,-1){20.00}}
\end{picture}}
}

\def\motylek{
{
\unitlength=.55pt
\begin{picture}(50.00,50.00)(0.00,23.00)
\thicklines
\qbezier(40.00,40.00)(60.00,25.00)(40.00,10.00)
\qbezier(10.00,40.00)(-10.00,25.00)(10.00,10.00)
\put(40.00,10.00){\makebox(0.00,0.00){$\bullet$}}
\put(40.00,40.00){\makebox(0.00,0.00){$\bullet$}}
\put(10.00,10.00){\makebox(0.00,0.00){$\bullet$}}
\put(10.00,40.00){\makebox(0.00,0.00){$\bullet$}}
\put(10.00,10.00){\line(0,-1){10.00}}
\put(10.00,10.00){\line(1,1){20.00}}
\put(40.00,40.00){\line(-1,-1){10.00}}
\put(40.00,50.00){\line(0,-1){10.00}}
\put(40.00,10.00){\line(0,-1){10.00}}
\put(30.00,20.00){\line(1,-1){10.00}}
\put(10.00,40.00){\line(1,-1){10.00}}
\put(10.00,50.00){\line(0,-1){10.00}}
\end{picture}}
}

\def\motylekfat{
{
\unitlength=.55pt
\begin{picture}(50.00,50.00)(0.00,23.00)
\thicklines
\qbezier(40.00,40.00)(60.00,25.00)(40.00,10.00)
\qbezier(10.00,40.00)(-10.00,25.00)(10.00,10.00)
\put(40.00,10.00){\makebox(0.00,0.00){$\bullet$}}
\put(40.00,40.00){\makebox(0.00,0.00){$\bullet$}}
\put(10.00,10.00){\makebox(0.00,0.00){$\bullet$}}
\put(10.00,40.00){\makebox(0.00,0.00){$\bullet$}}
\put(10.00,10.00){\line(0,-1){10.00}}
\put(10.00,10.00){\line(1,1){20.00}}
\put(40.00,40.00){\line(-1,-1){10.00}}
\put(40.00,50.00){\line(0,-1){10.00}}
\put(38,10.00){\line(0,-1){10.00}}
\put(41.50,10.00){\line(0,-1){10.00}}
\put(29,18.50){\line(1,-1){10.00}}
\put(31.50,21.50){\line(1,-1){10.00}}
\put(11.50,41.50){\line(1,-1){10.00}}
\put(9,39){\line(1,-1){10.00}}
\put(8.50,50.00){\line(0,-1){10.00}}
\put(11.50,50.00){\line(0,-1){10.00}}
\end{picture}}
}

\def\dvojiteypsilon{
\unitlength=.45pt
\thicklines
\begin{picture}(40.00,60.00)(0.00,28.00)
\put(20.00,20.00){\makebox(0.00,0.00){$\bullet$}}
\put(20.00,40.00){\makebox(0.00,0.00){$\bullet$}}
\put(20.00,20.00){\line(1,-1){20.00}}
\put(20.00,20.00){\line(-1,-1){20.00}}
\put(20.00,40.00){\line(0,-1){20.00}}
\put(20.00,40.00){\line(-1,1){20.00}}
\put(20.00,40.00){\line(1,1){20.00}}
\end{picture}
}

\def\muu{
\hskip .2em
\thicklines
{
\unitlength=.2pt
\begin{picture}(60.00,20.00)(0.00,10.00)
\put(30.00,30.00){\makebox(0.00,0.00){$\bullet$}}
\put(30.00,30.00){\line(1,-1){30.00}}
\put(0.00,0.00){\line(1,1){30.00}}
\put(30.00,30.00){\line(0,1){35}}
\end{picture}}
\hskip .2em
}

\def\Deltaa{
\hskip .2em
\thicklines
{
\unitlength=.2pt
\begin{picture}(60.00,20.00)(0.00,-50.00)
\put(30.00,-30.00){\makebox(0.00,0.00){$\bullet$}}
\put(30.00,-30.00){\line(1,1){30.00}}
\put(0.00,0.00){\line(1,-1){30.00}}
\put(30.00,-30.00){\line(0,-1){35}}
\end{picture}}
\hskip .2em
}

\def\lambdaa{
\hskip .2em
\thicklines
{
\unitlength=.2pt
\begin{picture}(60.00,20.00)(0.00,10.00)
\put(30.00,30.00){\makebox(0.00,0.00){$\bullet$}}
\put(33.00,33.00){\line(1,-1){30.00}}
\put(27.00,27.00){\line(1,-1){30.00}}
\put(0.00,0.00){\line(1,1){30.00}}
\put(34.00,30.00){\line(0,1){35}}
\put(26.00,30.00){\line(0,1){35}}
\end{picture}}
\hskip .2em
}

\def\muufat{
\hskip .2em
\thicklines
{
\unitlength=.2pt
\begin{picture}(60.00,20.00)(0.00,10.00)
\put(30.00,30.00){\makebox(0.00,0.00){$\bullet$}}
\put(33.00,33.00){\line(1,-1){30.00}}
\put(27.00,27.00){\line(1,-1){30.00}}
\put(-3.00,3.00){\line(1,1){30.00}}
\put(3.00,-3.00){\line(1,1){30.00}}
\put(34.00,30.00){\line(0,1){35}}
\put(26.00,30.00){\line(0,1){35}}
\end{picture}}
\hskip .2em
}

\def\Deltaafat{
\hskip .2em
\thicklines
{
\unitlength=.2pt
\begin{picture}(60.00,20.00)(0.00,-50.00)
\put(30.00,-30.00){\makebox(0.00,0.00){$\bullet$}}
\put(33.00,-33.00){\line(1,1){30.00}}
\put(27.00,-27.00){\line(1,1){30.00}}
\put(-3.00,-3.00){\line(1,-1){30.00}}
\put(3.00,3.00){\line(1,-1){30.00}}
\put(34.00,-30.00){\line(0,-1){35}}
\put(26.00,-30.00){\line(0,-1){35}}
\end{picture}}
\hskip .2em
}

\def\rhoo{
\hskip .2em
\thicklines
{
\unitlength=.2pt
\begin{picture}(60.00,20.00)(0.00,-50.00)
\put(30.00,-30.00){\makebox(0.00,0.00){$\bullet$}}
\put(30.00,-30.00){\line(1,1){30.00}}
\put(-3.00,-3.00){\line(1,-1){30.00}}
\put(3.00,3.00){\line(1,-1){30.00}}
\put(34.00,-30.00){\line(0,-1){35}}
\put(26.00,-30.00){\line(0,-1){35}}
\end{picture}}
\hskip .2em
}

\def\psii{
\hskip .2em
{
\thicklines
\unitlength=.35pt
\begin{picture}(40.00,40.00)(0.00,10.00)
\put(20.00,20.00){\makebox(0.00,0.00){$\bullet$}}
\put(0.00,40.00){\line(1,-1){40.00}}
\put(-2.00,2.00){\line(1,1){40.00}}
\put(2.00,-2.00){\line(1,1){40.00}}
\end{picture}}
\hskip .2em
}

\def\lii{
{
\unitlength=.4pt
\begin{picture}(80.00,30.00)(0.00,60.00)
\thicklines
\put(40.00,20.00){\makebox(0.00,0.00){$\bullet$}}
\put(60.00,40.00){\makebox(0.00,0.00){$\bullet$}}
\put(40.00,100.00){\makebox(0.00,0.00){$\bullet$}}
\put(20.00,80.00){\makebox(0.00,0.00){$\bullet$}}
\put(2.50,100.00){\line(0,1){20.00}}
\put(-2.50,100.00){\line(0,1){20.00}}
\put(77.50,20.00){\line(0,-1){20.00}}
\put(82.50,20.00){\line(0,-1){20.00}}
\put(-2.5,0){\qbezier(60.00,40.00)(80.00,30.00)(80.00,20.00)}
\put(2.5,0){\qbezier(60.00,43.00)(80.00,30.00)(80.00,20.00)}
\put(-2.5,0){\qbezier(20.00,80.00)(0.00,90.00)(0.00,100.00)}
\put(2.5,0){\qbezier(20.00,83.00)(0.00,90.00)(0.00,100.00)}
\qbezier(60.00,80.00)(60.00,70.00)(40.00,60.00)
\qbezier(20.00,40.00)(20.00,50.00)(40.00,60.00)
\qbezier(40.00,20.00)(20.00,30.00)(20.00,40.00)
\qbezier(40.00,100.00)(60.00,90.00)(60.00,80.00)
\put(-2.5,0){\qbezier(47.00,57.00)(60.00,50.00)(60.00,40.00)}
\put(2.5,0){\qbezier(49.00,57.00)(60.00,50.00)(60.00,40.00)}
\put(-2.5,0){\qbezier(20.00,80.00)(20.00,70.00)(31.00,63.00)}
\put(2.5,0){\qbezier(20.00,80.00)(20.00,70.00)(35.00,63.00)}
\put(40.00,20.00){\line(0,-1){20.00}}
\put(60.00,40.00){\line(-1,-1){20.00}}
\put(40.00,100.00){\line(-1,-1){20.00}}
\put(40.00,120.00){\line(0,-1){20.00}}
\end{picture}}
}

\def\hnida#1{
{
\unitlength=.5pt
\begin{picture}(40.00,40.00)(0.00,0.00)
\put(38,20){\makebox(0.00,0.00)[l]{$#1$}}
\put(24.00,2.00){\makebox(0.00,0.00){$...$}}
\put(24.00,38.00){\makebox(0.00,0.00){$...$}}
\put(20.00,20.00){\makebox(0.00,0.00){$\bullet$}}
\put(40.00,0.00){\line(-1,1){20.00}}
\put(10.00,0.00){\line(1,2){10.00}}
\put(0.00,0.00){\line(1,1){20.00}}
\put(20.00,20.00){\line(1,1){20.00}}
\put(20.00,20.00){\line(-1,2){10.00}}
\put(20.00,20.00){\line(-1,1){20.00}}
\end{picture}}
}

\begin{document}
\bibliographystyle{alpha}
\title{The $L_\infty$-deformation complex of diagrams of algebras}
\author[Y. Fr\'egier, M. Markl, D. Yau]{Yael Fr\'egier, Martin Markl,
  Donald Yau}
\thanks{The first author worked in the frame of grant F1R-MTH-PUL-08GEOQ
  of Professor Schlichenmaier.
The second author was supported by the grant GA \v CR
201/08/0397 and by the Academy of Sciences of the Czech Republic,
Institutional Research Plan No.~AV0Z10190503.}

\def\martin{\noindent{\bf Martin:\ }}
\def\endmartin{ \hfill\rule{10mm}{.75mm} \break}

\def\yael{\noindent{\bf Yael:\ }}
\def\endyael{ \hfill\rule{10mm}{.75mm} \break}


\email{yael.fregier@uni.lu}
\address{Mathematics Research Unit, 162A, avenue de la faiencerie  L-1511, Luxembourg,  Grand duchy of Luxembourg.}

\email{markl@math.cas.cz}
\address{Mathematical Institute of the Academy,
\v{Z}itn\'{a} 25, 115 67 Prague 1, The Czech Republic}

\email{dyau@math.ohio-state.edu}
\address{Department of Mathematics, The Ohio State University at Newark, 1179 University Drive, Newark, OH 43055, USA}

\date{\today}

\begin{abstract}
The deformation complex of an algebra over a colored PROP $\sfP$ is
defined in terms of a minimal (or, more generally, cofibrant)
model of $\sfP$. It is shown that it carries
the structure of an $L_\infty$-algebra which induces a graded Lie
bracket on cohomology.

As an example, the $L_\infty$-algebra structure on the deformation
complex of an associative algebra morphism $g$ is constructed.
Another example is the deformation complex of a Lie algebra
morphism. The last example is the diagram describing two mutually
inverse morphisms of vector spaces. Its $L_\infty$-deformation
complex has nontrivial $l_0$-term.

Explicit formulas for the $L_\infty$-operations in the above examples are
given. A typical deformation complex of a diagram
of algebras is a fully-fledged $L_\infty$-algebra with nontrivial
higher operations.
\end{abstract}

\maketitle

\tableofcontents

\sqsp

\section{Introduction}
\label{sec:intro}

In this paper, we construct the deformation complex $(C^*_\sfP(T;T),
\delta_\sfP)$ of an algebra $T$ over a colored PROP $\sfP$ and
observe that it has the structure of an $L_\infty$-algebra.
The cochain complex $(C^*_\sfP(T;T), \delta_\sfP)$ is so named because
its $L_\infty$-structure governs the deformations of $T$ in the form
of the \emph{Quantum Master Equation}
\eqref{eq:qme} (Section \ref{subsec:deformations}).

The existence of an $L_\infty$-structure on the deformation complex
of an algebra over an operad was proved in 2002 by van der Laan
\cite{van}. Van der Laan's construction was later generalized, in
\cite{mv}, to algebras over properads. The present paper will,
however, be based on the approach of the 2004 preprint
\cite{markl07b}.

Considering colored PROPs is necessary if one is to study 
$L_\infty$-deformations of, say, morphisms or more general diagrams of algebras
over a PROP, module-algebras, modules over an associative algebra, and
Yetter-Drinfel'd and Hopf modules over a bialgebra.  For example,
there is a $2$-colored PROP $\As_{\texttt{B} \to \texttt{W}}$ whose
algebras are of the form $f \colon U \to V$, in which $U$ and $V$ are
associative algebras and $f$ is a morphism of associative algebras
(Example \ref{ex:morphisms}).  Likewise, there is a $2$-colored PROP
$\ModAlg$ whose algebras are of the form $(H,A)$, in which $H$ is a
bialgebra and $A$ is an $H$-module-algebra (Example \ref{ex:modalg}).
Other examples of colored PROP algebras are given at the end of
Section~\ref{MAGDA}.

Here is a sketch of the construction of the deformation complex
$(C^*_\sfP(T;T), \delta_\sfP)$, with details given in Section \ref{sec:minimal}.  First we take a \emph{minimal model} (Definition \ref{def:minimal}) $\alpha \colon (\sfF(E), \partial) \to \sfP$ of the colored PROP $\sfP$, which should be thought of as a resolution of $\sfP$. Given a $\sfP$-algebra $\rho \colon \sfP \to \End_T$, we define
\[
C^*_\sfP(T;T) = \Der(\sfF(E), \mathcal{E}),
\]
in which $\mathcal{E} = \End_T$ is considered an $\sfF(E)$-module via the morphism $\beta = \rho\alpha$, and $\Der(\sfF(E), \mathcal{E})$ denotes
the vector space of derivations $\sfF(E) \to \mathcal{E}$.  The latter has a natural differential $\delta$ that sends $\theta \in \Der(\sfF(E),
\mathcal{E})$ to $\theta \partial$.


The $L_\infty$-operations on $C^*_\sfP(T;T)$ are constructed using graph substitutions (Section \ref{subsec:lk}).
 The usefulness of this very explicit construction of the $L_\infty$-operations on $C^*_\sfP(T;T)$ is first illustrated
 with the example of associative algebra morphisms.   For a morphism $g \colon U \to V$ of associative algebras, considered
 as an algebra over the $2$-colored PROP $\As_{\texttt{B} \to \texttt{W}}$, we are able to write down explicitly all the
 $L_\infty$-operations $l_k$ on the deformation complex of $g$ (Theorem \ref{thm:deltatheta} for $k = 1$, Theorem \ref{thm:l2}
 for $k = 2$, and Theorem \ref{thm:lktheta} for $k \geq 3$).  As expected, the underlying cochain complex of the deformation
  complex of $g$ is isomorphic to the Gerstenhaber-Schack\footnote{Be careful with the possible confusion with the complex associated to bialgebras. Both of these complexes share the same name} cochain
complex~\cite{gs83,gs85,gs88} of~$g$ (Theorem \ref{thm:deltatheta}).
   Therefore, the latter also has an explicit $L_\infty$-structure.  See Section \ref{subsec:back} for more discussion about this deformation complex.

A second example is given by the study of the case of Lie
algebra morphisms. As in the associative case,
there exists a
2-colored PROP $\Lie_{\texttt{B} \to \texttt{W}}$ whose 2-colored
algebras are morphism of
Lie algebras. We obtain then an explicit expression for the
$L_\infty$-operations (Theorem \ref{thm:deltathetaL} for $k = 1$,
Theorem \ref{thm:l2L} for $k = 2$, and Theorem \ref{thm:lkthetaL}
for $k \geq 3$). In particular the first operation $l_1$ gives a
complex isomorphic to the S-cohomology complex
\cite{fregier}. Hence this answers the
 question left open in \cite{fregier} of the existence of such an
 $L_\infty$-structure.

 Another example of the $L_\infty$-deformation complex
is given in Section \ref{MAGDA}. There is a $2$-colored operad
$\Iso$ (Example \ref{Iso}) whose algebras are of the form $F \colon
U \leftrightarrows V \colon G$, in which $U$ and $V$ are chain
complexes and $F$ and $G$ are mutually inverse chain maps. Using a
slight modification of the results and constructions of earlier
sections, we will write down explicitly the $L_\infty$-operations on
the deformation complex of a typical $\Iso$-algebra $T$ (Example
\ref{ex:LinftyIso}).

\noindent
{\bf Acknowledgment:} 
We would like to thank Jim Stasheff and Bruno Vallette for reading the
first version of the manuscript and many useful remarks.

\section{Preliminaries on colored PROPs}
\label{sec:prelim}

Fix a ground field $\bk$, assumed to be of characteristic $0$.  This
assumption is useful in considering models for operads or PROPs since
it guarantees the existence of the `averagization' of a
non-equivariant map into an equivariant one. The characteristic zero
assumption also simplifies concepts of Lie algebras and their
generalizations.

In this section, we review some basic definitions
about colored PROPs (and colored operads as their particular
instances), their algebras, colored $\Sigma$-bimodules, and free
colored PROPs.  Examples of algebras over colored PROPs can be found
at the end of this section.

\subsection{Colored $\Sigma$-bimodule}
\label{subsec:ColoredSigma}
Let $\frakC$ be a non-empty set whose elements are called \emph{colors}.  A \emph{$\Sigma$-bimodule} is a collection $E = \lbrace E(m,n) \rbrace_{m,n \geq 0}$ of $\bk$-modules in which each $E(m,n)$ is equipped with a left $\Sigma_m$ and a right $\Sigma_n$ actions that commute with each other.

A \emph{$\frakC$-colored $\Sigma$-bimodule} is a $\Sigma$-bimodule $E$ in which each $E(m,n)$ admits a $\frakC$-colored decomposition into submodules,
\begin{equation}
\label{eq:Edecomp}
E(m,n) = \bigoplus_{c_i,d_j \in \frakC} E\binom{d_1, \ldots , d_m}{c_1 , \ldots , c_n},
\end{equation}
that is compatible with the $\Sigma_m$-$\Sigma_n$-actions.
Elements of $E(m,n)$ are said to have \emph{biarity $(m,n)$}.
A \emph{morphism}
of $\frakC$-colored $\Sigma$-bimodules
is a linear bi-equivariant map
that respects the $\frakC$-colored decompositions \eqref{eq:Edecomp}.

$\frakC$-colored $\Sigma$-bimodules and their
morphism are examples of $\frakC$-colored objects; more examples will
follow. If $\frakC$ has $k$ elements, we will sometimes
call $\frakC$-colored objects simply {\em $k$-colored} objects.

\begin{definition}
\label{def:coloredPROP}
A \emph{$\frakC$-colored PROP} (\cite{maclane63,maclane65}, \cite[Section 8]{markl07}) is a $\frakC$-colored $\Sigma$-bimodule $\sfP = \lbrace \sfP(m,n) \rbrace$ (so each $\sfP(m,n)$ admits a $\frakC$-colored decomposition \eqref{eq:Edecomp}) that comes equipped with two operations: a \emph{horizontal composition}
\[
\otimes \colon
\sfP\binom{d_{11}, \ldots , d_{1m_1}}{c_{11}, \ldots , c_{1n_1}} \otimes \cdots \otimes
\sfP\binom{d_{s1}, \ldots , d_{sm_s}}{c_{s1}, \ldots , c_{sn_s}} \to
\sfP\binom{d_{11}, \ldots , d_{sm_s}}{c_{11}, \ldots , c_{sn_s}} \subseteq
\sfP(m_1 + \cdots + m_s, n_1 + \cdots + n_s)
\]
and a \emph{vertical composition}
\begin{equation}
\label{eq:vertcomp}
\circ \colon
\sfP\binom{d_1, \ldots , d_m}{c_1, \ldots , c_n} \otimes
\sfP\binom{b_1, \ldots , b_n}{a_1, \ldots , a_k} \to  \sfP\binom{d_1,
  \ldots , d_m}{a_1, \ldots , a_k} \subseteq \sfP(m,k), \quad
(x,y) \mapsto x \circ y.
\end{equation}

These two compositions are required to satisfy some associativity-type axioms.  There is also a unit element $\texttt{1}_c \in \sfP\binom{c}{c}$ for each color $c$.  Moreover, the vertical composition $x \circ y$ in \eqref{eq:vertcomp} is $0$, unless
\[
c_i = b_i \quad \text{for} \quad 1 \leq i \leq n.
\]
\emph{Morphisms} of $\frakC$-colored PROPs are
unit-preserving morphisms of the underlying $\Sigma$-bimodules that commute with both the horizontal and the vertical compositions.
\end{definition}

{\em Colored operads\/} are particular
cases of colored PROPs such that $\sfP\binom{d_1, \ldots , d_m}{a_1,
  \ldots , a_k} = 0$ for $m \geq 2$.
Note that colored PROPs can also be defined as
ordinary ($1$-colored) PROPs over the semisimple algebra $K = \oplus_{c \in \frakC} \bk_c$, where each $\bk_c$ is a copy of the ground field $\bk$ \cite[Section 2]{markl04}.

\begin{example}
\label{ex:EndPROP}
The \emph{$\frakC$-colored endomorphism PROP} $\End^\frakC_T$ of a
 $\frakC$-graded module $T = \oplus_{c \in \frakC} T_c$ is the $\frakC$-colored PROP with
\[
\End^\frakC_T\binom{d_1, \ldots , d_m}{c_1 , \ldots , c_n}
= \Hom_\bk(T_{c_1} \otimes \cdots \otimes T_{c_n}, T_{d_1} \otimes \cdots \otimes T_{d_m}).
\]
The horizontal composition is given by tensor products of $\bk$-linear maps.  The vertical composition is given by composition of $\bk$-linear maps with matching colors.
\end{example}

\begin{definition}
For a $\frakC$-colored PROP $\sfP$, a \emph{$\sfP$-algebra} is a morphism of $\frakC$-colored PROPs
\[
\alpha \colon \sfP \to \End^\frakC_T
\]
for some $\frakC$-graded module $T = \oplus_{c \in \frakC} T_c$.  In this case, we say that $T$ is a $\sfP$-algebra.
\end{definition}

\subsection{Pasting scheme for $\frakC$-colored PROPs}
\label{subsec:scheme}

For $m,n \geq 1$, let $\UGr^\frakC(m,n)$ be the set whose elements are pairs $(G, \zeta)$ such that:

\begin{enumerate}
\item $G \in \UGr(m,n)$ is a directed $(m,n)$-graph \cite[p.38]{markl07}.
\item For each vertex $v \in Vert(G)$, the sets $out(v)$ (outgoing edges from $v$) and $in(v)$ (incoming edges to $v$) are labeled $1, \ldots , q$ and $1, \ldots , p$, respectively, where $\# out(v) = q$ and $\# in(v) = p$.
\item $\zeta \colon edge(G) \to \frakC$ is a function that assigns to each edge in $G$ a color in $\frakC$.  For any edge $l \in edge(G)$, $\zeta(l) \in \frakC$ is called the \emph{color of $l$}.
\end{enumerate}
There is a $\frakC$-colored decomposition
\[
\UGr^\frakC(m,n) = \coprod_{c_i, d_j \in \frakC} \UGr^\frakC\binom{d_1, \ldots , d_m}{c_1, \ldots , c_n},
\]
where $(G,\zeta) \in \UGr^\frakC\binom{d_1, \ldots , d_m}{c_1, \ldots , c_n}$ if and only if the input legs $\lbrace l^1_{in}, \ldots l^n_{in} \rbrace$ of $G$ have colors $c_1, \ldots , c_n$ and the output legs $\lbrace l^1_{out}, \ldots , l^m_{out}\rbrace$ of $G$ have colors $d_1, \ldots , d_m$.

As in \cite{markl07}, $\UGr^\frakC(m,n)$ and $\UGr^\frakC\binom{d_1, \ldots , d_m}{c_1, \ldots , c_n}$ are categories with color-respecting isomorphisms as morphisms.  Elements in $\UGr^\frakC(m,n)$ are called \emph{$\frakC$-colored directed $(m,n)$-graphs}.

\subsection{Decoration on colored directed graphs}
\label{subsec:decor}

Let $E$ be a $\frakC$-colored $\Sigma$-bimodule and $(G,\zeta)$ be a $\frakC$-colored directed $(m,n)$-graph.  Define
\begin{equation}
\label{eq:E(G)}
E(G,\zeta) = \bigotimes_{v \in Vert(G)} E \binom{\zeta(o_v^1) , \ldots , \zeta(o_v^q)}{\zeta(i^v_1) , \ldots , \zeta(i^v_p)},
\end{equation}
where $in(v) = \lbrace i^v_1, \ldots , i^v_p \rbrace$ and $out(v) =
\lbrace o_v^1, \ldots , o_v^q \rbrace$.  Its elements are called
\emph{$E$-decorated $\frakC$-colored directed $(m,n)$-graphs}.

For an element $\Gamma = \otimes_v e_v \in E(G,\zeta)$, the element $e_v \in E\binom{\zeta(o_v^1) , \ldots , \zeta(o_v^q)}{\zeta(i^v_1) , \ldots , \zeta(i^v_p)}$ corresponding to the vertex $v \in Vert(G)$ is called the \emph{decoration of $v$}.

In other words, $E(G,\zeta)$ is the space of decorations of the vertices of the $\frakC$-colored directed $(m,n)$-graph $(G,\zeta)$ with elements of $E$ with matching biarity and colors.

\subsection{Free colored PROP}
\label{subsec:free}

Let $E$ be a $\frakC$-colored $\Sigma$-bimodule.  For $c_1, \ldots , c_n, d_1, \ldots , d_m \in \frakC$, define the module
\[
\sfF^\frakC(E)\binom{d_1, \ldots , d_m}{c_1 , \ldots , c_n}
= \colim E(G,\zeta),
\]
where the colimit is taken over the category $\UGr^\frakC\binom{d_1, \ldots , d_m}{c_1, \ldots , c_n}$.
Then
\[
\sfF^\frakC(E) = \left\lbrace \sfF^\frakC(E)(m,n) = \bigoplus_{c_i, d_j \in \frakC} \sfF^\frakC(E)\binom{d_1, \ldots , d_m}{c_1, \ldots , c_n} \right\rbrace
\]
is a $\frakC$-colored PROP, in which the horizontal composition $\otimes$ is given by disjoint union of $E$-decorated $\frakC$-colored directed $(m,n)$-graphs.  The vertical composition $\circ$ in $\sfF^\frakC(E)$ is given by grafting of $\frakC$-colored legs with matching colors.

Note that there is a natural $\mathbb{Z}_{\geq 0}$-grading,
\[
\sfF^\frakC(E) = \bigoplus_{k \geq 0} \sfF^\frakC_k(E),
\]
where $\sfF^\frakC_k(E)$ is the submodule generated by the monomials involving $k$ elements in $E$.

\begin{proposition}[= $\frakC$-colored version of Proposition 57 in \cite{markl07}]
$\sfF^\frakC(E) = \lbrace \sfF^\frakC(E)(m,n) \rbrace$  is the free $\frakC$-colored PROP generated by the $\frakC$-colored $\Sigma$-bimodule $E$.  In other words, the functor $\sfF^\frakC$ is the left adjoint of the forgetful functor from $\frakC$-colored PROPs to $\frakC$-colored $\Sigma$-bimodules.
\end{proposition}

In particular, elements in the free $\frakC$-colored PROP $\sfF^\frakC(E)$ can be written as sums of $E$-decorated $\frakC$-colored directed graphs.

\begin{convention}
From now on, everything will be tacitly assumed to be $\frakC$-colored with a
suitable set of colors $\frakC$.  When there is no danger of ambiguity, we will, for brevity, suppress $\frakC$ from the notation.
\end{convention}

\begin{example}[Morphisms]
\label{ex:morphisms}
Let $\sfP$ be an ordinary PROP (i.e., a $1$-colored PROP).  Then there is a $2$-colored PROP $\sfP_{\texttt{B} \to \texttt{W}}$ whose algebras are of the form $f \colon U \to V$, in which $U$ and $V$ are $\sfP$-algebras and $f$ is a morphism of $\sfP$-algebras \cite[Example 1]{markl02}.  It can be constructed as the quotient
\[
\sfP_{\texttt{B} \to \texttt{W}} =
\frac{\sfP_{\texttt{B}} \ast \sfP_{\texttt{W}} \ast
  \sfF(f)}{\left(f^{\otimes m} x_{\texttt{B}} = x_{\texttt{W}}
  f^{\otimes n} \text{ for all } x \in
\sfP(m,n)\right)},
\]
where $\sfP_{\texttt{B}}$ and $\sfP_{\texttt{W}}$ are copies of $\sfP$
concentrated in the colors $\texttt{B}$ and $\texttt{W}$,
respectively, $x_{\texttt{B}}$ and $x_{\texttt{W}}$ are the respective
copies of $x$ in $\sfP_{\texttt{B}}$ and $\sfP_{\texttt{W}}$,
and $\sfF(f)$ is the free $2$-colored PROP on the generator $f \colon
\texttt{B} \to \texttt{W}$.  The star $\ast$ denotes the free product
($=$ the coproduct) of $2$-colored PROPs.

In the case that $\sfP$ is the operad $\As$ for associative algebras,
cohomology of $\As_{\texttt{B} \to \texttt{W}}$-algebras (i.e.,
associative algebra morphisms) will be discussed in details in
Sections \ref{sec:algmor} and \ref{sec:higherbracket}.
\end{example}

\begin{example}[Modules]
\label{ex:modules}
There is a $2$-colored operad $\AsMod$ whose algebras are of the form $(A,M)$, where $A$ is an associative algebra and $M$ is a left $A$-module.  It can be constructed as the quotient
\[
\AsMod = \frac{\sfF(\mu,\lambda)}{\left(\mu(\mu \otimes 1_\texttt{A}) - \mu(1_\texttt{A} \otimes \mu), \lambda(\mu \otimes 1_\texttt{M}) - \lambda(1_\texttt{A} \otimes \lambda)\right)}.
\]
Here $\sfF(\mu,\lambda)$ is the free $2$-colored operad (with $\frakC = \lbrace \texttt{A}, \texttt{M} \rbrace$) on the generators,
\[
\mu \in \sfF(\mu,\lambda)\binom{\texttt{A}}{\texttt{A},\texttt{A}} \quad \text{and} \quad
\lambda \in \sfF(\mu,\lambda)\binom{\texttt{M}}{\texttt{A},\texttt{M}},
\]
which encode the multiplication in $A$ and the left $A$-action on $M$,
respectively.

If we depict the multiplication $\mu$ as $\muu$ and the module
action $\lambda$ as $\lambdaa$, then the associativity of $\mu$ is
expressed by the diagram
\begin{equation}
\label{mm1}
\ZZbbZbZ = \ZbZbbZZ
\end{equation}
and the compatibility between the multiplication and the module action by
\begin{equation}
\label{mm2}
\BfZZbbZbZ = \BfZbZbbZZ\ .
\end{equation}
The diagrams in the above two displays should be interpreted as
elements of the free colored PROP $\sfF(\mu,\lambda)$, with the
$\texttt{A}$-colored edges of the underlying graph represented by
simple lines $\hskip .2em \thicklines\unitlength
1em\thicklines\line(0,1){1}\hskip .2em$, and the $\texttt{M}$-colored edges by
the double lines $\hskip .2em \thicklines\unitlength
1em\line(0,1){1}\hskip .15em \line(0,1){1}\hskip .2em$. We use the
convention that the directed edges point upwards, i.e.~the composition
is performed from the bottom up.
\end{example}

\begin{example}[Module-algebras]
\label{ex:modalg}
Let $H = (H, \mu_H, \Delta_H)$ be a (co)associative bialgebra.  An $H$-\emph{module-algebra} is an associative algebra $(A,\mu_A)$ that is equipped with a left $H$-module structure such that the multiplication map on $A$ becomes an $H$-module morphism.  In other words, the \emph{module-algebra axiom}
\[
x(ab) = \sum_{(x)} \,(x_{(1)}a)(x_{(2)}b)
\]
holds for $x \in H$ and $a, b \in A$, where $\Delta_H(x) = \sum_{(x)} x_{(1)} \otimes x_{(2)}$ using the Sweedler's notation for comultiplication.

This algebraic structure arises often in algebraic topology \cite{boardman}, quantum groups \cite{kassel}, Lie and Hopf algebras theory \cite{d,mont,sweedler}, and group representations \cite{abe}.  For example, in algebraic topology, the complex cobordism ring $\mathrm{MU}^*(X)$ of a topological space $X$ is an $S$-module-algebra, where $S$ is the Landweber-Novikov algebra \cite{landweber,novikov} of stable cobordism operations.

Another important example of a module-algebra arises in the theory of Lie algebras.  Finite dimensional simple $sl(2, \mathbb{C})$-modules are, up to isomorphism, the highest weight modules $V(n)$ $(n \geq 0)$ \cite[Theorem 7.2]{humphreys}.  There is a $U(sl(2,\mathbb{C}))$-module-algebra structure on the polynomial algebra $\mathbb{C} \lbrack x,y \rbrack$ such that the submodule $\mathbb{C} \lbrack x,y \rbrack_n$ of homogeneous polynomials of degree $n$ is isomorphic to the highest weight module $V(n)$ \cite[Theorem V.6.4]{kassel}.  In other words, all the finite dimensional simple $sl(2, \mathbb{C})$-modules can be encoded inside a single $U(sl(2,\mathbb{C}))$-module-algebra.

There is a $2$-colored PROP $\ModAlg$ whose algebras are of the form
$(H,A)$, where $H$ is a (co)associative bialgebra and $A$ is an
$H$-module-algebra.  (The third author first learned about this fact
from Bruno Vallette in private correspondence.)  It can be constructed
as the quotient (with $\frakC = \lbrace \texttt{H},\texttt{A} \rbrace$)

\begin{equation}
\label{eq:ModAlg}
\ModAlg = \sfF(\mu_\texttt{H}, \Delta_\texttt{H}, \mu_\texttt{A}, \lambda)/I,
\end{equation}
where $\sfF = \sfF(\mu_\texttt{H}, \Delta_\texttt{H}, \mu_\texttt{A}, \lambda)$ is the free $2$-colored PROP on the generators:
\[
\mu_\texttt{H} \in \sfF\binom{\texttt{H}}{\texttt{H},\texttt{H}}, \,
\Delta_\texttt{H} \in \sfF\binom{\texttt{H},\texttt{H}}{\texttt{H}}, \,
\mu_\texttt{A} \in \sfF\binom{\texttt{A}}{\texttt{A},\texttt{A}}, \, \text{and }
\lambda \in \sfF\binom{\texttt{A}}{\texttt{H},\texttt{A}},
\]
which encode the multiplication and comultiplication in $H$, the multiplication in $A$, and the $H$-module structure on $A$, respectively.  The ideal $I$ is generated by the elements:
\[
\begin{split}
& \mu_\texttt{H}(\mu_\texttt{H} \otimes 1_\texttt{H}) - \mu_\texttt{H}(1_\texttt{H} \otimes \mu_\texttt{H}) \quad (\text{associativity of } \mu_\texttt{H}), \\
& (\Delta_\texttt{H} \otimes 1_\texttt{H})\Delta_\texttt{H} - (1_\texttt{H} \otimes \Delta_\texttt{H})\Delta_\texttt{H}  \quad (\text{coassociativity of } \Delta_\texttt{H}), \\
& \Delta_\texttt{H} \mu_\texttt{H} -  \mu_\texttt{H}^{\otimes 2} (2 ~ 3) \Delta_\texttt{H}^{\otimes 2} \quad (\text{compatibility of } \mu_\texttt{H} \text{ and } \Delta_\texttt{H}), \\
& \mu_\texttt{A}(\mu_\texttt{A} \otimes 1_\texttt{A}) - \mu_\texttt{A}(1_\texttt{A} \otimes \mu_\texttt{A}) \quad (\text{associativity of } \mu_\texttt{A}), \\
& \lambda(\mu_\texttt{H} \otimes 1_\texttt{A}) - \lambda(1_\texttt{H} \otimes \lambda) \quad (H\text{-module axiom}), \\
& \lambda(1_\texttt{H} \otimes \mu_\texttt{A}) - \mu_\texttt{A} \lambda^{\otimes 2} (2 ~ 3) (\Delta_\texttt{H} \otimes 1_\texttt{A}^{\otimes 2}) \quad  (\text{module-algebra axiom}).
\end{split}
\]
Here $(2 ~ 3) \in \Sigma_4$ is the permutation that switches $2$ and $3$.

If we draw the multiplication $\mu_H$ as $\muu$, the comultiplication
$\Delta_H$ as $\Deltaa$, the multiplication $\mu_A$ as $\muufat$, and
the $H$-module action $\lambda$ as $\lambdaa$, then the
bialgebra axioms for $H$ are expressed by
\[
\ZZbbZbZ = \ZbZbbZZ,\
\CCbbCbC = \CbCbbCC\
\mbox { and }\
\dvojiteypsilon = \hskip .4em \motylek \hskip .3em,
\]
The associativity of $\mu_H$ is given by the obvious $\hskip .2em
\thicklines\unitlength 1em\line(0,1){1}\hskip .15em \line(0,1){1}\hskip
.2em$-colored version of~(\ref{mm1}), the $H$-module axiom
by~(\ref{mm2}), and the module-algebra axiom by
\[
\BfZbZbbZZa = \hskip .3em \GerdaI \hskip 1em .
\]

Variants of module-algebras, including, module-co/bialgebras and
comodule-(co/bi)algebras are algebras over similar $2$-colored PROPs.
Deformations, in the classical sense \cite{ger64}, of module-algebras
and its variants were studied in \cite{yau07,yau08}.
\end{example}

\begin{example}[Entwining structures]
\label{ex:entwining}
An \emph{entwining structure} \cite{brez01,bm} is a tuple
$(A, C, \psi)$, in which $A = (A, \mu)$ is an associative algebra, $C = (C, \Delta)$ is a coassociative coalgebra, and $\psi \colon C \otimes A \to A \otimes C$, such that the following two entwining axioms are satisfied:
\begin{equation}
\label{eq:entaxioms}
\begin{split}
\psi(\Id_C \otimes \mu)
&= (\mu \otimes \Id_C)(\Id_A \otimes \psi)(\psi \otimes \Id_A), \\
(\Id_A \otimes \Delta)\psi
&= (\psi \otimes \Id_C)(C \otimes \psi)(\Delta \otimes \Id_A).
\end{split}
\end{equation}
If we symbolize $\mu$ by $\muu$, $\Delta$ by $\Deltaafat$ and $\psi$ by
$\psii$, then the entwining axioms can be written as
\[
\raisebox{-1.2em}{\rule{0pt}{0pt}}
\GerdaIIinv = \GerdaIIIinv\  \mbox { and }\
\GerdaII = \GerdaIII \hskip .2em.
\]

This algebraic structure arises in the study of \emph{coalgebra-Galois extension} and its dual notion, \emph{algebra-Galois coextension} \cite{bh}, generalizing the Hopf-Galois extension of \cite{kt}.

There is a $2$-colored PROP $\Ent$ whose algebras are entwining structures.  It can be constructed as the quotient
\[
\Ent = \sfF(\mu,\Delta,\psi)/I
\]
of the free $2$-colored PROP $\sfF = \sfF(\mu,\Delta,\psi)$ (with $\frakC = \lbrace \texttt{A},\texttt{C} \rbrace$) on the generators:
\[
\mu \in \sfF\binom{\texttt{A}}{\texttt{A},\texttt{A}}, \quad
\Delta \in \sfF\binom{\texttt{C},\texttt{C}}{\texttt{C}}, \quad \text{and} \quad
\psi \in \sfF\binom{\texttt{A},\texttt{C}}{\texttt{C},\texttt{A}}.
\]
The ideal $I$ is generated by the elements expressing the associativity of $\mu$, the coassociativity of $\Delta$, and the two entwining axioms \eqref{eq:entaxioms}.
\end{example}

\begin{example}[Yetter-Drinfel'd modules]
\label{ex:yetter}
A \emph{Yetter-Drinfel'd module} \cite{yetter} (a.k.a.\ \emph{crossed bimodule} and \emph{quantum Yang-Baxter module}) over a (co)associative bialgebra $(H, \mu, \Delta)$ is a vector space $M$ together with a left $H$-module action $\omega \colon H \otimes M \to M$ and a right $H$-comodule coaction $\rho \colon M \to M \otimes H$ that satisfy the Yetter-Drinfel'd condition,
   \begin{equation}
   \label{eq:YDaxiom}
   (\Id_M \otimes \mu) \circ (\rho \otimes \Id_H) \circ \tau \circ (\Id_H \otimes \omega) \circ (\Delta \otimes \Id_M)
   =\, (\omega \otimes \mu) \circ (\Id_H \otimes \tau \otimes \Id_H) \circ (\Delta \otimes \rho),
   \end{equation}
where $\tau$ is the twist isomorphism $H \otimes M \cong M \otimes
H$. If we depict $\mu$ as $\muu$, $\Delta$ as $\Deltaa$, $\omega$ as
$\lambdaa$, and $\rho$ as $\rhoo$, then
\[
\lii = \motylekfat \hskip .3em.
\raisebox{-2em}{\rule{0pt}{0pt}}
\]

Yetter-Drinfel'd modules were introduced by Yetter \cite{yetter}, and
are studied further in \cite{lr,ps,radford,rt,sch}, among others.  If
the bialgebra $H$ is a finite dimensional Hopf algebra, then the
left-modules over its Drinfel'd double $D(H)$ are exactly the
Yetter-Drinfel'd modules over $H$.  These objects play important roles
in the theory of quantum groups and mathematical physics.  Indeed, a
finite dimensional Yetter-Drinfel'd module $M$ gives rise to a
solution of the quantum Yang-Baxter equation \cite{lr,radford} (i.e.,
an $R$-matrix \cite[Chapter VIII]{kassel}).  Conversely, through the
so-called \emph{FRT construction} \cite{frt,kassel}, every $R$-matrix
on a finite dimensional vector space gives rise to a Yetter-Drinfel'd
module over some bialgebra.  Cohomology for Yetter-Drinfel'd modules
and their morphisms over a fixed bialgebra have been studied in
\cite{ps} and \cite{yau08b}, respectively.

There is a $2$-colored PROP $\YD$ whose algebras are of the form $(H,M)$, where $H$ is a bialgebra and $M$ is a Yetter-Drinfel'd module over $H$.  It can be constructed as the quotient
\[
\YD = \sfF(\mu, \Delta, \omega, \rho)/I
\]
of the free $2$-colored PROP $\sfF = \sfF(\mu, \Delta, \omega, \rho)$ (with $\frakC = \lbrace \texttt{H},\texttt{M} \rbrace$) on the generators:
\[
\mu \in \sfF\binom{\texttt{H}}{\texttt{H},\texttt{H}}, ~
\Delta \in \sfF\binom{\texttt{H},\texttt{H}}{\texttt{H}}, ~
\omega \in \sfF\binom{\texttt{M}}{\texttt{H},\texttt{M}}, ~ \text{and} ~
\rho \in \sfF\binom{\texttt{M},\texttt{H}}{\texttt{M}}.
\]
The ideal $I$ is generated by elements expressing the bialgebra axioms for $\mu$ and $\Delta$, the left $H$-module axiom for $\omega$, the right $H$-comodule axiom for $\rho$, and the Yetter-Drinfel'd condition \eqref{eq:YDaxiom}.
\end{example}

\begin{example}[Hopf modules]
\label{ex:hopf}
A \emph{Hopf module} over a bialgebra $(H,\mu,\Delta)$ is a vector space $M$ together with a left $H$-module action $\omega \colon H \otimes M \to M$ and a right $H$-comodule coaction $\rho \colon M \to M \otimes H$ that satisfy the Hopf module condition:
\begin{equation}
\label{eq:hopfmod}
\rho \circ \omega \,=\, (\omega \otimes \mu) \circ (\Id_H \otimes \tau \otimes \Id_H) \circ (\Delta \otimes \rho).
\end{equation}
There is a $2$-colored PROP $\HopfMod$ whose algebras are of the form $(H,M)$, in which $H$ is a bialgebra and $M$ is a Hopf module over $H$.  It admits the same construction as $\YD$, except that the Yetter-Drinfel'd condition \eqref{eq:YDaxiom} is replaced by the Hopf module condition \eqref{eq:hopfmod} in the ideal $I$ of relations.
\end{example}

\section{Minimal models and cohomology}
\label{sec:minimal}

In this section, we define (i) minimal models for colored PROPs and
(ii) cohomology for algebras over a colored PROP based on minimal
models. Since minimal models are (at least in most cases) known to be
unique up to isomorphism, the cohomology based on minimal models is
unique already on the chain level. We will
therefore require minimality whenever possible, though an arbitrary
cofibrant resolution should give the same cohomology. There are,
however, important PROPs that do not have a minimal model, as the
colored operad $\Iso$ considered in Section~\ref{MAGDA}.

First we recall the notions of modules and derivations for colored PROPs.

\subsection{Modules}
\label{subsec:modules}

For a $\frakC$-colored PROP $\sfP$ and a $\frakC$-colored $\Sigma$-bimodule $U$, a \emph{$\sfP$-module} structure on $U$ \cite[p.203]{markl96b} consists of the following operations:
\[
\begin{split}
\circ = \circ_l & \colon \sfP\binom{d_1, \ldots , d_m}{c_1, \ldots , c_n} \otimes U\binom{b_1, \ldots , b_n}{a_1, \ldots , a_k} \to U\binom{d_1, \ldots , d_m}{a_1, \ldots , a_k}, \\
\circ = \circ_r & \colon U\binom{d_1, \ldots , d_m}{c_1, \ldots , c_n} \otimes \sfP\binom{b_1, \ldots , b_n}{a_1, \ldots , a_k} \to U\binom{d_1, \ldots , d_m}{a_1, \ldots , a_k}, \\
\otimes = \otimes_l & \colon \sfP\binom{d_1, \ldots , d_{m_1}}{c_1, \ldots , c_{n_1}} \otimes U\binom{b_1, \ldots , b_{m_2}}{a_1, \ldots , a_{n_2}} \to U\binom{d_1, \ldots , d_{m_1}, b_1, \ldots , b_{m_2}}{c_1, \ldots , c_{n_1}, a_1, \ldots , a_{n_2}}, \\
\otimes = \otimes_r & \colon U\binom{d_1, \ldots , d_{m_1}}{c_1, \ldots , c_{n_1}} \otimes \sfP\binom{b_1, \ldots , b_{m_2}}{a_1, \ldots , a_{n_2}} \to U\binom{d_1, \ldots , d_{m_1}, b_1, \ldots , b_{m_2}}{c_1, \ldots , c_{n_1}, a_1, \ldots , a_{n_2}}.
\end{split}
\]
As usual, the vertical operations $\circ_l$ and $\circ_r$ are trivial unless $b_i = c_i$ for $1 \leq i \leq n$.  The following compatibility axioms are also imposed on the four operations:
\[
\begin{split}
f \circ (g \circ h) &= (f \circ g) \circ h, \\
f \otimes (g \otimes h) &= (f \otimes g) \otimes h, \\
(f_1 \circ f_2) \otimes (g_1 \circ g_2) &= (f_1 \otimes g_1) \circ (f_2 \otimes g_2).
\end{split}
\]
Here exactly one of $f, g$, and $h$ lies in $U$ and the other two lie in $\sfP$.  Likewise, exactly one of $f_1$, $f_2$, $g_1$, and $g_2$ lies in $U$ and the other three lie in $\sfP$.

We note that $\sfP$-modules can also be defined as
abelian group objects in the category ${\tt PROP}/\sfP$ of $\frakC$-colored PROPs over
$\sfP$.

For example, if $\beta \colon \sfP \to \mathsf{Q}$ is a morphism of $\frakC$-colored PROPs, then $\mathsf{Q}$ becomes a $\sfP$-module via $\beta$ in the obvious way.

\subsection{Derivations}
\label{subsec:der}

Given a $\sfP$-module $U$, a \emph{derivation} $P \to U$ is a $\frakC$-colored $\Sigma$-bimodule morphism $d \colon P \to U$ that satisfies the usual derivation property with respect to both the vertical operations $\circ$ and the horizontal operations $\otimes$ \cite[p.204]{markl96b}.  Denote by $\Der(\sfP,U)$ the vector space of derivations $\sfP \to U$.  

\begin{proposition}[= $\frakC$-colored version of Proposition 3 in \cite{markl07b}]
\label{prop:Der}
Let $U$ be an $\sfF^\frakC(E)$-module for some $\frakC$-colored $\Sigma$-bimodule $E$.  Then there is a canonical isomorphism
\begin{equation}
\label{eq:Deriso}
\Der(\sfF^\frakC(E), U) \cong \Hom_\Sigma^\frakC(E,U),
\end{equation}
where $\Hom_\Sigma^\frakC(E,U)$ denotes the vector space of $\frakC$-colored $\Sigma$-bimodule morphisms $E \to U$.
\end{proposition}

In one direction, the isomorphism \eqref{eq:Deriso} takes a derivation $\theta \in \Der(\sfF^\frakC(E), U)$ to its restriction $\theta \vert_E$ to the space $E$ of generators. In the other direction, it takes a map $\varphi \colon E \to U \in \Hom_\Sigma^\frakC(E,U)$ to its unique extension $Ex(\varphi) \colon \sfF^\frakC(E) \to U$ as a derivation such that $Ex(\varphi) \vert_E = \varphi$.

Following \cite{markl96,markl06}, we make the following definition.

\begin{definition}
\label{def:minimal}
Let $\sfP$ be a $\frakC$-colored PROP.   A \emph{minimal model} of $\sfP$ is a differential graded $\frakC$-colored PROP $(\sfF^\frakC(E), \partial)$ for some $\frakC$-colored $\Sigma$-bimodule $E$ together with a homology isomorphism
\[
\rho \colon (\sfF^\frakC(E), \partial) \to (\sfP, 0)
\]
such that the following \emph{minimality condition} is satisfied:
\[
\partial(E) \subseteq \bigoplus_{k \geq 2} \sfF^\frakC_k(E).
\]
In other words, the image of $E$ under $\partial$ consists of decomposables.
\end{definition}



\subsection{Cohomology}
\label{subsec:cohomology}

Here we define cohomology of an algebra over a colored PROP following \cite{markl96b,markl07b}.

Let $\sfP$ be a $\frakC$-colored PROP, and let $(\sfF^\frakC(E), \partial) \xrightarrow{\rho} (\sfP, 0)$ be a minimal model of $\sfP$.  Let $\sfP \xrightarrow{\alpha} \End^\frakC_T$ be a $\sfP$-algebra structure on $T = \oplus_{c \in \frakC} T_c$.  Consider $\End^\frakC_T$ as an $\sfF^\frakC(E)$-module via the morphism
\[
\beta = \alpha \rho \colon \sfF^\frakC(E) \to \End^\frakC_T.
\]
Then the map
\[
\begin{split}
\Der(\sfF^\frakC(E), \End^\frakC_T) & \xrightarrow{\delta} \Der(\sfF^\frakC(E), \End^\frakC_T) \\
\theta & \mapsto \theta \partial
\end{split}
\]
is well-defined and is a differential $(\delta^2 = 0)$ because $\partial^2 = 0$.

\begin{definition}
\label{def:defcomplex}
In the above setting, define the cochain complex
\begin{equation}
\label{eq:CP}
C^*_\sfP(T;T) ~=~ \uparrow \Der(\sfF^\frakC(E), \End^\frakC_T)^{-*},
\end{equation}
where the degree $+1$ differential $\delta_\sfP$ is induced by $\delta$, $\uparrow$ denotes suspension, and $-*$ denotes reversed grading.  We call $(C^*_\sfP(T;T), \delta_\sfP)$ the \emph{deformation complex of $T$}.  Its cohomology,
\[
H^*_\sfP(T;T) = H(C^*_\sfP(T;T), \delta_P),
\]
is called the \emph{cohomology of $T$ with coefficients in itself}.
\end{definition}

Note that if $\frakC = \lbrace * \rbrace$, i.e., $\sfP$ is an ordinary ($1$-colored) PROP, then $(C^*_\sfP(T;T), \delta_\sfP)$ and $H^*_\sfP(T;T)$ defined above coincide with the definitions in \cite{markl96b,markl07b}.

\section{$L_\infty$-structure on $C^*_\sfP(T;T)$ and deformations}
\label{sec:Linfty}

In this section, we observe that the deformation complex $(C^*_\sfP(T;T), \delta_\sfP)$ (Definition \ref{def:defcomplex}) of an algebra $T$ over a colored PROP $\sfP$ has the natural structure of an $L_\infty$-algebra (Theorem \ref{thm:Linfinity}).  The relationship between this $L_\infty$-algebra and deformations of $T$ is discussed in Section \ref{subsec:deformations}.   An explicit construction of the $L_\infty$-operations $l_k$ in $(C^*_\sfP(T;T), \delta_\sfP)$ is given in Section \ref{subsec:lk}.  This construction will first be applied in Sections \ref{sec:algmor} and \ref{sec:higherbracket} to obtain very explicit formulas for the operations $l_k$ in the deformation complex of an associative algebra morphism.

First we recall the notion of an $L_\infty$-algebra.

\begin{definition}[Definition 2.1 in \cite{ladamarkl95}, Example 3.90 in \cite{mss}]
An \emph{$L_\infty$-structure} on a $\mathbb{Z}$-graded module $V$ consists of a sequence of operations $(\delta = l_1, l_2, l_3, \ldots)$ with
\[
l_n \colon V^{\otimes n} \to V
\]
of degree $2-n$ such that each $l_n$ is anti-symmetric and the condition
\begin{equation}
\label{eq:ln}
\sum_{i+j=n+1} \sum_\sigma \chi(\sigma) (-1)^{i(j-1)} l_j\left(l_i(x_{\sigma(1)}, \ldots , x_{\sigma(i)}), x_{\sigma(i+1)}, \ldots , x_{\sigma(n)}\right) = 0
\end{equation}
holds for $n \geq 1$.  Here $\sigma$ runs through all the $(i,n-i)$-unshuffles for $i \geq 1$, and
\[
\chi(\sigma) = sgn(\sigma) \cdot \varepsilon(\sigma; x_1, \ldots , x_n),
\]
where $\varepsilon(\sigma; x_1, \ldots , x_n)$ is the Koszul sign given by
\[
x_1 \wedge \cdots \wedge x_n = \varepsilon(\sigma; x_1, \ldots , x_n) \cdot x_{\sigma(1)} \wedge \cdots \wedge x_{\sigma(n)}.
\]
In this case, we call $(V, \delta, l_2, l_3, \ldots)$ an \emph{$L_\infty$-algebra}.  The anti-symmetry of $l_n$ means that
\[
l_n(x_{\sigma(1)}, \ldots , x_{\sigma(n)}) = \chi(\sigma)l_n(x_1, \ldots , x_n)
\]
for $\sigma \in \Sigma_n$ and $x_1, \ldots , x_n \in V$.
\end{definition}

\begin{theorem}
\label{thm:Linfinity}
In the setting of \S \ref{subsec:cohomology},
there exists an $L_\infty$-structure $(\delta_\sfP, l_2, l_3, \ldots)$
on $C^*_\sfP(T;T)$ capturing deformations of colored $\sfP$-algebras
in the sense of~\ref{subsec:deformations} below.
This $L_\infty$-structure induces a graded Lie
algebra structure on $H^*_\sfP(T;T)$.
\end{theorem}

\begin{proof}
This is the $\frakC$-colored version of \cite[Theorem 1]{markl07b}, whose proof, with some very minor modifications, applies to the $\frakC$-colored setting as well.  In fact, Sections 3 and 4 in \cite{markl07b} (which contain the proof of Theorem 1 in that paper) apply basically verbatim to the $\frakC$-colored setting.  An explicit ``graphical" construction of the operations $l_k$ will be given below (\S \ref{subsec:lk}).
\end{proof}

\subsection{Deformations of colored PROP algebras}
\label{subsec:deformations}

Section 5 in \cite{markl07b} concerning deformations of algebras over a PROP also applies to the $\frakC$-colored setting without change.  In particular, deformations of an algebra $T$ over a $\frakC$-colored PROP $\sfP$ (i.e., $\sfF^\frakC(E)$-algebra structures on $T$) correspond to elements $\kappa \in C^1_\sfP(T;T)$ that satisfy the \emph{Quantum Master Equation} \cite[Eq.(4)]{markl07b}:
\begin{equation}
\label{eq:qme}
0 = \delta_\sfP(\kappa) + \frac{1}{2!}l_2(\kappa, \kappa) - \frac{1}{3!}l_3(\kappa, \kappa, \kappa) - \frac{1}{4!}l_4(\kappa, \kappa, \kappa, \kappa) + \cdots.
\end{equation}
In other words, the $L_\infty$-algebra
\[
(C^*_\sfP(T;T), \delta_\sfP, l_2, l_3, \ldots)
\]
in Theorem \ref{thm:Linfinity} controls the deformations of $T$ as a $\sfP$-algebra.  As explained in \cite[Introduction]{markl07b}, this $L_\infty$-algebra is an $L_\infty$-version of the Deligne groupoid \cite{hinich97,hinich04} governing deformations that are described by the usual \emph{Master Equation} (also known as the \emph{Maurer-Cartan Equation}):
\[
0 = d\kappa + \frac{1}{2} \lbrack \kappa, \kappa \rbrack.
\]

When $\sfP$ is a properad \cite{vallette}, there is another approach to studying
 $L_\infty$-deformations of $\sfP$-algebras due to Merkulov and Vallette \cite{mv}.
   Their approach is based on a generalization of Van der Laan's \emph{homotopy (co)operads} \cite{van}
   to \emph{homotopy (co)properads}.  They show that the deformation complex $(C^*_\sfP(T;T), \delta_\sfP)$
   inherits a $L_\infty$-algebra structure from a homotopy properad  (Theorem 28 of  \cite{mv}).
     Vallette recently informed the third author in private correspondence that the paper \cite{mv} can also be
     extended to the colored setting.

\subsection{Construction of the operations $l_k$ on $C^*_\sfP(T;T)$}
\label{subsec:lk}

Here we describe how the operations $l_k$ in Theorem \ref{thm:Linfinity} are constructed, again following \cite[Section 2]{markl07b} closely.

Suppose that $F_1, \ldots , F_k \in \Hom^\frakC_\Sigma(E,
\End^\frakC_T)$ and that $\Gamma \in E(G,\zeta)$ is an $E$-decorated
$\frakC$-colored directed $(m,n)$-graph \eqref{eq:E(G)} with
underlying $\frakC$-colored graph $(G,\zeta) \in \UGr^\frakC(m,n)$.
Let $v_1, \ldots , v_k \in Vert(G)$ be $k$ distinct vertices in $G$.
Consider the $\End^\frakC_T$-decorated $\frakC$-colored directed
$(m,n)$-graph
\[
\Gamma^{\lbrace v_1, \ldots , v_k \rbrace}_{\lbrace\beta\rbrace}
\lbrack F_1, \ldots , F_k \rbrack \in \End^\frakC_T(G,\zeta)
\]
obtained from $\Gamma$ by:
\begin{enumerate}
\item replacing the decoration $e_{v_i} \in E$ of the vertex $v_i$ by $F_i(e_{v_i}) \in \End^\frakC_T$ for $1 \leq i \leq k$, and
\item replacing the decoration $e_v \in E$ of any vertex $v \not\in \lbrace v_1, \ldots , v_k \rbrace$ by $\beta(e_v) = \alpha\rho(e_v)$.
\end{enumerate}
The graph $\Gamma^{\lbrace v_1, \ldots , v_k \rbrace}_{\lbrace\beta\rbrace}
\lbrack F_1, \ldots , F_k \rbrack$ is visualized in
Figure~\ref{za_chvili_volam_Jarce} which is a colored version of a
picture taken from~\cite{markl07b}.
\begin{figure}[t]
\begin{center}
{
\unitlength=1pt
\begin{picture}(170.00,170.00)(0.00,0.00)
\thicklines
\put(90.00,10.00){\makebox(0.00,0.00){$\cdots$}}
\put(90.00,160.00){\makebox(0.00,0.00){$\cdots$}}
\put(130.00,80.00){\makebox(0.00,0.00){$\hnida {\beta}$}}
\put(70.00,40.00){\makebox(0.00,0.00){$\hnida {\beta}$}}
\put(50.00,70.00){\makebox(0.00,0.00){$\hnida {\beta}$}}
\put(130.00,130.00){\makebox(0.00,0.00){$\hnida {F_k}$}}
\put(100.00,60.00){\makebox(0.00,0.00){$\hnida {F_2}$}}
\put(73.00,100.00){\makebox(0.00,0.00){$\hnida {F_3}$}}
\put(50.00,120.00){\makebox(0.00,0.00){$\hnida {F_1}$}}
\put(140.00,150.00){\vector(0,1){20.00}}
\put(60.00,150.00){\vector(0,1){20.00}}
\put(50.00,150.00){\vector(0,1){20.00}}
\put(140.00,0.00){\vector(0,1){20.00}}
\put(60.00,0.00){\vector(0,1){20.00}}
\put(50.00,0.00){\vector(0,1){20.00}}
\put(95.00,85.00){\oval(130.00,130.00)}
\end{picture}}
\end{center}
\caption{\label{za_chvili_volam_Jarce} 
The $\End^\frakC_T$-decorated graph $\Gamma_{\{\beta\}}^{\{\Rada
v1k\}}[\Rada F1k]$. Vertices labelled $F_i$ are decorated by
$F_i(e_{v_i})$, $1 \leq i \leq k$, the remaining vertices are decorated
by $\beta(e_v)$.}
\end{figure}
Using the $\frakC$-colored PROP structure on $\End^\frakC_T$ (Example \ref{ex:EndPROP}), the graph $\Gamma^{\lbrace v_1, \ldots , v_k \rbrace}_{\lbrace\beta\rbrace} \lbrack F_1, \ldots , F_k \rbrack$ produces an element
\[
\gamma\left(\Gamma^{\lbrace v_1, \ldots , v_k \rbrace}_{\lbrace\beta\rbrace} \lbrack F_1, \ldots , F_k \rbrack\right) \in \End^\frakC_T\binom{\zeta(l_{out}^1), \ldots , \zeta(l_{out}^m)}{\zeta(l_{in}^1), \ldots , \zeta(l_{in}^n)} \subseteq \End^\frakC_T(m,n).
\]
Here $\lbrace l_{out}^1, \ldots , l_{out}^m \rbrace$ and $\lbrace
l_{in}^1, \ldots , l_{in}^n \rbrace$ are the output and input legs,
respectively, of $(G,\zeta)$.

Now pick cochains $f_1, \ldots , f_k \in C^*_\sfP(T;T)$, which correspond to $F_1, \ldots , F_k \in \Hom^\frakC_\Sigma(E, \End^\frakC_T)$ under the isomorphism \eqref{eq:Deriso}:
\begin{equation}
\label{eq:caniso}
C^*_\sfP(T;T) = ~ \uparrow\Der(\sfF^\frakC(E), \End^\frakC_T)^{-*}
\cong~ \uparrow\Hom_\Sigma^\frakC(E,\End^\frakC_T)^{-*}.
\end{equation}
If $\xi \in E\binom{d_1, \ldots , d_m}{c_1, \ldots , c_n}$, then $\partial(\xi) \in \sfF^\frakC(E)\binom{d_1, \ldots , d_m}{c_1, \ldots , c_n}$ can be written as a finite sum
\[
\partial(\xi) = \sum_{s \in S_\xi} \Gamma_s,
\]
where each $\Gamma_s \in E(G,\zeta)$ for some $(G,\zeta) \in \UGr^\frakC\binom{d_1, \ldots , d_m}{ c_1, \ldots , c_n}$.  Define
\[
l_k(f_1, \ldots , f_k)(\xi) \in \End^\frakC_T\binom{d_1, \ldots , d_m}{c_1, \ldots , c_n}
\]
to be the element
\begin{equation}
\label{eq:lkf}
l_k(f_1, \ldots , f_k)(\xi)
\buildrel \text{def} \over=
(-1)^{\nu(f_1, \ldots , f_k)} \sum_{s \in S_\xi} \sum_{(v_1, \ldots , v_k)} \gamma\left(\Gamma^{\lbrace v_1, \ldots , v_k \rbrace}_{s, \lbrace\beta\rbrace} \lbrack F_1, \ldots , F_k \rbrack\right),
\end{equation}
where $(v_1, \ldots , v_k)$ runs through all the $k$-tuples of distinct vertices in the underlying graph of $\Gamma_s$.  The sign on the right-hand side of \eqref{eq:lkf} is given by
\begin{equation}
\label{eq:nu}
\nu(f_1, \ldots , f_k)
\buildrel \text{def} \over=
(k-1)\vert f_1 \vert + (k-2)\vert f_2 \vert + \cdots + \vert f_{k-1} \vert.
\end{equation}
Since $\xi$ is arbitrary, \eqref{eq:lkf} specifies an element
\begin{equation}
\label{eq:lk'}
l_k(f_1, \ldots , f_k) \in \Hom_\Sigma^\frakC(E,\End^\frakC_T) \cong C^*_\sfP(T;T).
\end{equation}

The arguments in Sections 3-4 in \cite{markl07b} ensure that
\eqref{eq:lk'} is indeed well-defined.  We note that the $L_\infty$
axiom \eqref{eq:ln} for the operations $l_k$ constructed above is a
consequence of $\partial^2 = 0$.  Also, an obvious modification of the
above construction applies to free cofibrant, not necessarily minimal,
models as well.  We will see an instance of such a generalization in
Section \ref{MAGDA}.

\section{Deformation complex of an associative algebra morphism}
\label{sec:algmor}

In this section and Section \ref{sec:higherbracket}, we illustrate the
$L_\infty$-deformation theory of colored PROP algebras (Section
\ref{sec:Linfty}) in the case of associative algebra morphisms.  Let
$g \colon U \to V$ be an associative algebra morphism, and set $T = U
\oplus V$ as a $2$-colored graded module.  Let $\As_{\texttt{B} \to \texttt{W}}$ denote
the $2$-colored operad encoding associative algebra morphisms (Example \ref{ex:morphisms}).  The morphism $g \colon U \to V$ can be regarded as an $\As_{\texttt{B} \to \texttt{W}}$-algebra structure on $T$.

The purposes of this section are (i)
 to express the differential $\delta_{\As_{\texttt{B} \to
 \texttt{W}}}$ in $C^*_{\As_{\texttt{B} \to \texttt{W}}}(T;T)$
 \eqref{eq:CP} in terms of the Hochschild differential (Theorem
 \ref{thm:deltatheta}), and (ii)
to observe that the cochain complex $(C^*_{\As_{\texttt{B} \to
 \texttt{W}}}(T;T), \delta_{\As_{\texttt{B} \to \texttt{W}}})$
is isomorphic to the Gerstenhaber-Schack cochain complex
 $(C^{*+1}_{GS}(g;g), d_{GS})$ of the morphism $g$
 \cite{gs83,gs85,gs88}
(Theorem \ref{thm:deltatheta}).
This isomorphism allows us to transfer the $L_\infty$-structure on
 $(C^*_{\As_{\texttt{B} \to \texttt{W}}}(T;T), \delta_{\As_{\texttt{B}
 \to \texttt{W}}})$ to the
Gerstenhaber-Schack cochain complex $(C^{*+1}_{GS}(g;g), d_{GS})$ (Corollary \ref{cor:LinfGS}).

The materials in this section and Section \ref{sec:higherbracket} can be easily dualized to obtain an explicit $L_\infty$-structure on the deformation complex of a morphism of coassociative coalgebras.  The associated deformation theory of coalgebra morphisms is the one constructed in \cite{yau}.

\subsection{Background}
\label{subsec:back}

Deformation of an associative algebra morphism $g$, in the classical sense of Gerstenhaber \cite{ger64}, was studied by Gerstenhaber and Schack in \cite{gs83,gs85,gs88}.  In the case of a single associative algebra $A$, the deformation complex is the Hochschild cochain complex $C^*(A;A)$ of $A$ with coefficients in itself, which has the structure of a differential graded Lie algebra \cite{ger63}.  On the other hand, the work of Gerstenhaber and Schack \cite{gs83,gs85,gs88} left open the question of what structure the deformation complex $(C^*_{GS}(g;g), d_{GS})$ of $g$ possesses.  Borisov answered this question in \cite{borisov} by showing that $(C^*_{GS}(g;g), d_{GS})$, while not a differential graded Lie algebra, is isomorphic to the underlying cochain complex of an $L_\infty$-algebra.

With our approach based on minimal models, we are able to write down all the $L_\infty$-operations $l_k$ on $C^*_{\As_{\texttt{B} \to \texttt{W}}}(T;T)$ explicitly (with $l_1$ in Theorem \ref{thm:deltatheta} and $l_k$ ($k \geq 2$) in Section \ref{sec:higherbracket}).  In particular, all the higher $l_k$ $(k \geq 3)$ can be written in terms of a certain generalized ``comp" operation \eqref{eq:circa'}, which extends the classical $\circ_i$ operation in the Hochschild cochain complex \cite{ger63}.  As far as we know, these higher operations $l_k$ have never been explicitly written down before.  We believe that this example of associative algebra morphisms will serve as a guide for obtaining explicit formulas for the $L_\infty$-operations in the deformation complexes of other kinds of morphisms and general diagrams.

\subsection{The Gerstenhaber-Schack complex $(C^*_{GS}(g;g), d_{GS})$}
\label{subsec:CGS}

Here we recall the Gerstenhaber-Schack cochain complex $(C^*_{GS}(g;g), d_{GS})$ \cite{gs83,gs85,gs88}.

Fix a morphism $g \colon U \to V$ of associative algebras.  We also
consider $V$ as a $U$-bimodule via $g$.  Then
\begin{equation} \label{eq:CnGS} C^n_{GS}(g;g) \buildrel \text{def}
\over= \Hom(U^{\otimes n}, U) \oplus \Hom(V^{\otimes n}, V) \oplus
\Hom(U^{\otimes n-1}, V)
\end{equation}
for $n \geq 1$.  A typical element in $C^n_{GS}(g;g)$ is denoted by $(x_U,x_V,x_g)$ with $x_U \in \Hom(U^{\otimes n}, U)$, $x_V \in \Hom(V^{\otimes n}, V)$, and $x_g \in \Hom(U^{\otimes n-1}, V)$.  Its differential is defined as
\[
d^n_{GS}(x_U, x_V, x_g) \buildrel \text{def} \over=
(bx_U, bx_V, gx_U - x_Vg^{\otimes n} - bx_g),
\]
where $b$ denotes the Hochschild differential in $\Hom(U^{\otimes *}, U)$, $\Hom(V^{\otimes *}, V)$, or $\Hom(U^{\otimes *},V)$.

\subsection{The minimal model of $\As_{\texttt{B} \to \texttt{W}}$}
\label{subsec:minimalAs}

Here we recall from \cite{markl02,markl04} the minimal model of the $2$-colored operad $\As_{\texttt{B} \to \texttt{W}}$ that encodes associative algebra morphisms.

The $2$-colored operad $\As_{\texttt{B}\to \texttt{W}}$ can be represented as (Example \ref{ex:morphisms})
\[
\As_{\texttt{B}\to \texttt{W}}
= \frac{\As_\texttt{B} \ast \As_\texttt{W} \ast \sfF(f)}{(f\mu = \nu f^{\otimes 2})},
\]
where $\mu$ and $\nu$ denote the generators in $\As_\texttt{B}(1,2)$ and $\As_\texttt{W}(1,2)$, respectively, which encode the multiplications in the domain and the target.

Let $E$ be the $2$-colored $\Sigma$-bimodule with the following generators:
\[
\begin{split}
\mu_n & \colon \texttt{B}^{\otimes n} \to \texttt{B} \text{ of degree } n-2 \text{ and biarity } (1,n) ~(n \geq 2), \\
\nu_n &\colon \texttt{W}^{\otimes n} \to \texttt{W} \text{ of degree } n-2 \text{ and biarity } (1,n) ~(n \geq 2), \text{ and}\\
f_n & \colon \texttt{B}^{\otimes n} \to \texttt{W} \text{ of degree } n-1 \text{ and biarity } (1,n) ~(n \geq 1).
\end{split}
\]
Then the minimal model for $\As_{\texttt{B}\to \texttt{W}}$ is
\[
(\sfF(E), \partial) \xrightarrow{\alpha} \As_{\texttt{B}\to \texttt{W}},
\]
where
\begin{equation*}
\alpha(\mu_n) =
\begin{cases} \mu & \text{if } n = 2, \\ 0 & \text{otherwise} \end{cases},\,
\alpha(\nu_n) = \begin{cases} \nu & \text{if } n = 2, \\ 0 & \text{otherwise} \end{cases},
\end{equation*}
and
\begin{equation*}
\alpha(f_n) = \begin{cases} f & \text{if } n = 1, \\ 0 & \text{otherwise} \end{cases}.
\end{equation*}

The differential $\partial$ is given by:
\begin{subequations}
\allowdisplaybreaks
\begin{align}
\partial(\mu_n) & = \sum_{\substack{i+j\,=\,n+1 \\ i,j \,\geq\, 2}} \sum_{s=0}^{n-j} (-1)^{i + s(j+1)} \mu_i \circ_{s+1} \mu_j, \label{eq:dmu} \\
\partial(\nu_n) & = \sum_{\substack{i+j\,=\,n+1 \\ i,j \,\geq\, 2}} \sum_{s=0}^{n-j} (-1)^{i + s(j+1)} \nu_i \circ_{s+1} \nu_j, \label{eq:dnu} \\
\partial(f_n) & = -\sum_{l=2}^n \sum_{r_1 + \cdots + r_l = n} (-1)^{\sum_{1 \leq i < j \leq l} r_i(r_j + 1)} \nu_l(f_{r_1} \otimes \cdots \otimes f_{r_l}) \notag \\
 & \relphantom{} - \sum_{\substack{i+j\,=\,n+1 \\ i \,\geq\, 1,\, j \,\geq\, 2}} \sum_{s=0}^{n-j} (-1)^{i+s(j+1)} f_i \circ_{s+1} \mu_j.  \label{eq:df}
\end{align}
\end{subequations}
Here
\begin{equation}
\label{eq:circ}
\mu_i \circ_{s+1} \mu_j
\buildrel \text{def} \over=
\mu_i\left(1_\texttt{B}^{\otimes s} \otimes \mu_j \otimes 1_\texttt{B}^{\otimes i-s-1}\right),
\end{equation}
which ``plugs" $\mu_j$ into the $(s+1)$st input of $\mu_i$ (see
Figure~\ref{Prijde_dnes_Jarka?}), and similarly for $\nu_i \circ_{s+1}
\nu_j$ and $f_i \circ_{s+1} \mu_j$.
\begin{figure}[t]
{
\unitlength=1.000000pt
\begin{picture}(60.00,70.00)(0.00,0.00)
\thicklines
\put(35.00,23.00){\makebox(0.00,0.00)[l]{\scriptsize $(s+1)$th input}}
\put(32.00,-5.00){\makebox(0.00,0.00){$\cdots$}}
\put(22.00,30.00){\makebox(0.00,0.00){$\cdots$}}
\put(42.00,30.00){\makebox(0.00,0.00){$\cdots$}}
\put(30.00,15.00){\makebox(0.00,0.00){$\bullet$}}
\put(26.00,15.00){\makebox(0.00,0.00)[rb]{\scriptsize $\mu_j$}}
\put(30.00,50.00){\makebox(0.00,0.00){$\bullet$}}
\put(26.00,50.00){\makebox(0.00,0.00)[rb]{\scriptsize $\mu_i$}}
\put(30.00,15.00){\line(3,-2){30.00}}
\put(30.00,15.00){\line(-1,-1){20.00}}
\put(30.00,15.00){\line(-3,-2){30.00}}
\put(30.00,50.00){\line(0,-1){35.00}}
\put(30.00,50.00){\line(3,-2){30.00}}
\put(30.00,50.00){\line(-1,-1){20.00}}
\put(30.00,50.00){\line(-3,-2){30.00}}
\put(30.00,50.00){\line(0,1){20.00}}
\end{picture}}
\caption{\label{Prijde_dnes_Jarka?}
The graph corresponding the composition $\mu_i \circ_{s+1} \mu_j$.}
\end{figure}
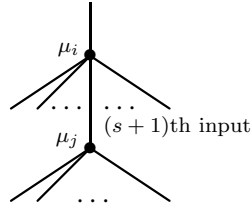

\subsection{The cochain complex $(C^*_{\As_{\texttt{B}\to \texttt{W}}}(T;T), \delta_{\As_{\texttt{B}\to \texttt{W}}})$}

Suppose that $T = U \oplus V$ as a $2$-colored graded module and that $g \colon U \to V$ is a morphism of associative algebras represented by the morphism $\rho \colon \As_{\texttt{B}\to \texttt{W}} \to \End_T$ of $2$-colored operads.  Then the canonical isomorphism \eqref{eq:caniso} says in this case,
\[
C^*_{\As_{\texttt{B}\to \texttt{W}}}(T;T)
=~ \uparrow\Der(\sfF(E), \End_T)^{-*}
\cong~ \uparrow\Hom_\Sigma^{\lbrace \texttt{B},\texttt{W} \rbrace}(E, \End_T)^{-*}.
\]
Under this isomorphism, an element $\theta \in C^n_{\As_{\texttt{B}\to
 \texttt{W}}}(T;T)$ is uniquely determined by the tuple
\begin{equation}
\label{eq:thetaiso}
\left(\theta_U, \theta_V, \theta_g\right)
\buildrel \text{def} \over=
\left(\theta(\mu_{n+1}), \theta(\nu_{n+1}), \theta(f_n)\right) \in  C^{n+1}_{GS}(g;g).
\end{equation}
This establishes a linear isomorphism
\begin{align*}
C^n_{\As_{\texttt{B}\to \texttt{W}}}(T;T) & \cong  C^{n+1}_{GS}(g;g), \\
\theta & \leftrightarrow  \left(\theta_U, \theta_V, \theta_g\right).
\end{align*}

Denote by $\delta_{GS}$ the differential on the graded module $C^*_{GS}(g;g)$ induced by $\delta_{\As_{\texttt{B}\to \texttt{W}}}$.  The identification \eqref{eq:thetaiso} provides an isomorphism
\[
(C^*_{\As_{\texttt{B}\to \texttt{W}}}(T;T), \delta_{\As_{\texttt{B}\to \texttt{W}}}) \xrightarrow{\cong} (C^{*+1}_{GS}(g;g), \delta_{GS})
\]
of cochain complexes.

\begin{theorem}
\label{thm:deltatheta}
For $\theta \in C^{n-1}_{\As_{\texttt{B}\to \texttt{W}}}(T;T)$, we have
\begin{equation}
\label{eq:deltatheta}
\delta_{GS}\left(\theta_U, \theta_V, \theta_g\right)
= \left((-1)^{n+1} b\theta_U, (-1)^{n+1}b \theta_V, g\theta_U - \theta_V g^{\otimes n} - (-1)^n b\theta_g\right),
\end{equation}
in which $b$ denotes the appropriate Hochschild differential.  In particular, there is a cochain complex isomorphism
\begin{equation}
\label{eq:CASiso}
\left(C^{*-1}_{\As_{\texttt{B}\to \texttt{W}}}(T;T), \delta_{\As_{\texttt{B} \to \texttt{W}}}\right) \cong \left(C^*_{GS}(g;g), d_{GS}\right)
\end{equation}
given by
\[
\theta \in C^{n-1}_{\As_{\texttt{B}\to \texttt{W}}}(T;T) \mapsto
\left((-1)^{\frac{n(n+1)}{2}} \theta_U, (-1)^{\frac{n(n+1)}{2}} \theta_V, (-1)^{\frac{(n-1)n}{2}}\theta_g\right).
\]
\end{theorem}

Since $(C^*_{\As_{\texttt{B}\to \texttt{W}}}(T;T), \delta_{\As_{\texttt{B}\to \texttt{W}}} = l_1, l_2, l_3, \ldots)$ is an $L_\infty$-algebra (Theorem \ref{thm:Linfinity}), we can use the cochain complex isomorphism \eqref{eq:CASiso} to transfer the higher brackets $l_k$ $(k \geq 2)$ to $(C^{*+1}_{GS}(g;g), d_{GS})$.

\begin{corollary}
\label{cor:LinfGS}
There is an $L_\infty$-algebra structure $(d_{GS} = l_1, l_2, l_3,
\ldots)$ on $C^{*+1}_{GS}(g;g)$ governing deformations of the
associative algebra morphism $g$.
\end{corollary}

\begin{proof}[Proof of Theorem \ref{thm:deltatheta}]
Since $\delta_{\As_{\texttt{B}\to \texttt{W}}} = l_1$ in the $L_\infty$-algebra and since the degree of $l_1$ is $+1$, we have
\[
\delta_{GS}\left(\theta_U, \theta_V, \theta_g\right) = (l_1(\theta)(\mu_{n+1}), l_1(\theta)(\nu_{n+1}), l_1(\theta)(f_n))
\]
by the identification \eqref{eq:thetaiso}.  Therefore, to prove \eqref{eq:deltatheta}, it suffices to show:
\begin{equation}
\label{eq:deltathetal1}
\begin{split}
l_1(\theta)(\mu_{n+1}) &= (-1)^{n+1} b\theta_U, \\
l_1(\theta)(\nu_{n+1}) &= (-1)^{n+1}b \theta_V, \,\text{and} \\
l_1(\theta)(f_n) &= g\theta_U - \theta_V g^{\otimes n} - (-1)^n b\theta_g.
\end{split}
\end{equation}

From the description \eqref{eq:lkf} of the operation $l_k$, the
computation of $l_1(\theta)(\mu_{n+1})$ starts with
$\partial(\mu_{n+1})$ \eqref{eq:dmu}. As an $E$-decorated $2$-colored
directed $(1,n+1)$-graph, the term $\mu_i \circ_{s+1} \mu_j$ in
$\partial(\mu_{n+1})$ has two vertices, whose decorations are $\mu_i$
and $\mu_j$, see Figure~\ref{Prijde_dnes_Jarka?}.
Therefore, the expression~\eqref{eq:lkf}, when applied
to the current situation, gives
\begin{equation}
\label{eq:l1theta}
l_1(\theta)(\mu_{n+1}) =
\sum_{\substack{i+j\,=\,n+2 \\ i,j \,\geq\,2}} \sum_{s=0}^{n+1-j} (-1)^{i+s(j+1)} \left\lbrace\theta(\mu_i) \circ_{s+1} \beta(\mu_j) + \beta(\mu_i) \circ_{s+1} \theta(\mu_j)\right\rbrace.
\end{equation}
Note that, since $\theta \in C^{n-1}_{\As_{B\to W}}(T;T)$,
\[
\theta(\mu_i) = \begin{cases} 0 & \text{if }i \not=n, \\ \theta_U & \text{if } i=n,\end{cases}
\quad \text{and} \quad
\beta(\mu_j) = \rho(\alpha(\mu_j)) = \begin{cases} 0 & \text{if } j \not= 2, \\ \mu_U & \text{if } j = 2,\end{cases}
\]
where $\mu_U \colon U^{\otimes 2} \to U$ is the multiplication on $U$.  It follows that \eqref{eq:l1theta} reduces to
\[
\begin{split}
l_1(\theta)(\mu_{n+1})
&= \sum_{s=0}^{n-1} (-1)^{n+s(2+1)} \theta_U \circ_{s+1} \mu_U
+ \sum_{s=0}^1 (-1)^{2+s(n+1)} \mu_U \circ_{s+1} \theta_U \\
&= (-1)^{n+1} \mu_U(-, \theta_U) + \mu_U(\theta_U,-)
+ (-1)^{n+1} \sum_{s=1}^n (-1)^s \theta_U(\Id^{\otimes s-1}_U \otimes \mu_U \otimes \Id^{\otimes n-s}_U) \\
&= (-1)^{n+1} b\theta_U,
\end{split}
\]
which is the first condition in \eqref{eq:deltathetal1}.

The previous paragraph applies verbatim to $l_1(\theta)(\nu_{n+1})$ (with $\nu_l$ replacing $\mu_l$ everywhere), since the definition of $\partial(\nu_*)$ \eqref{eq:dnu} admits the same formula as that of $\partial(\mu_*)$.  Therefore, it remains to show the last condition in \eqref{eq:deltathetal1}.

In $\partial(f_n)$ \eqref{eq:df}, the term $\nu_l(f_{r_1} \otimes \cdots \otimes f_{r_l})$ (respectively, $f_i \circ_{s+1} \mu_j$) is an $E$-decorated $2$-colored directed $(1,n)$-graph with $l+1$ (respectively, $2$) vertices.  Since
\[
\beta(f_j) = \rho(\alpha(f_j)) = \begin{cases} g \colon U \to V & \text{if } j = 1, \\ 0 & \text{otherwise},\end{cases}
\]
the same kind of analysis as above gives
\[
\begin{split}
l_1(\theta)(f_n)
&= -\theta_Vg^{\otimes n} - (-1)^{(n-1)(1+1)}\mu_V(\theta_g \otimes g) - (-1)^{n-1+1} \mu_V(g \otimes \theta_g) \\
& \relphantom{} - \sum_{s=0}^{n-2} (-1)^{n-1+s(2+1)} \theta_g \circ_{s+1} \mu_U - (-1)^{1 +0}g \theta_U.\\
&= g\theta_U - \theta_Vg^{\otimes n} - (-1)^n\left\lbrace\mu_V(g \otimes \theta_g) + (-1)^n \mu_V(\theta_g \otimes g) + \sum_{s=1}^{n-1} (-1)^s \theta_g \circ_{s} \mu_U\right\rbrace  \\
&= g\theta_U - \theta_Vg^{\otimes n} - (-1)^n b\theta_g. \\
\end{split}
\]
Here $\mu_V \colon V^{\otimes 2} \to V$ denotes the multiplication on $V$.  This establishes the last condition in \eqref{eq:deltathetal1} and finishes the proof of Theorem \ref{thm:deltatheta}.
\end{proof}

\section{The higher brackets in $C^*_{\As_{\texttt{B}\to \texttt{W}}}(T;T)$}
\label{sec:higherbracket}

We keep the same setting and notations as in the previous section.  The purpose of this section is to make explicit the $L_\infty$-operations $l_k$ on $C^*_{\As_{\texttt{B}\to \texttt{W}}}(T;T)$ for $k \geq 2$.  The cases $k = 2$ (Theorem \ref{thm:l2}) and $k \geq 3$ (Theorem \ref{thm:lktheta}) are treated separately.  As an immediate consequence of our explicit formula for $l_k$ $(k \geq 3)$, we observe that, when applied to the tensor powers of $C^{\leq q}_{\As_{\texttt{B}\to \texttt{W}}}(T;T)$ for some fixed $q \geq 0$, only $\delta_{\As_{\texttt{B}\to \texttt{W}}}(T;T) = l_1, l_2, \ldots , l_{q+2}$ can be non-trivial (Corollary \ref{cor:lktheta}).

\subsection{The operation $l_2$}

First we deal with the case $k = 2$.  Pick elements $\theta \in C^{n-1}_{\As_{\texttt{B}\to \texttt{W}}}(T;T)$ and $\omega \in C^{m-1}_{\As_{\texttt{B}\to \texttt{W}}}(T;T)$.  Under the identification \eqref{eq:thetaiso}, $\theta$ and $\omega$ correspond to
\[
(\theta_U,\theta_V,\theta_g) \in C^n_{GS}(g;g) \quad \text{and} \quad
(\omega_U,\omega_V,\omega_g) \in C^m_{GS}(g;g),
\]
respectively.

Since $l_2$ has degree $0$, the element $l_2(\theta,\omega)$ lies in $C^{(n+m-1)-1}_{\As_{\texttt{B}\to \texttt{W}}}(T;T)$.  Under the identification \eqref{eq:thetaiso}, $l_2(\theta,\omega)$ is uniquely determined by
\[
(l_2(\theta,\omega)(\mu_{n+m-1}), l_2(\theta,\omega)(\nu_{n+m-1}), l_2(\theta,\omega)(f_{n+m-2})) \in C^{n+m-1}_{GS}(g;g).
\]

\begin{theorem}
\label{thm:l2}
With the notations above, we have
\begin{subequations}
\allowdisplaybreaks
\begin{align}
l_2(\theta,\omega)(\mu_{n+m-1})
&= -\sum_{s=1}^n (-1)^{(s+1)(m+1)} \theta_U \circ_s \omega_U
- (-1)^{n+m} \sum_{s=1}^m (-1)^{(s+1)(n+1)} \omega_U \circ_s \theta_U,\label{eq:l2U}\\
l_2(\theta,\omega)(\nu_{n+m-1})
&= - \sum_{s=1}^n (-1)^{(s+1)(m+1)} \theta_V \circ_s \omega_V
- (-1)^{n+m} \sum_{s=1}^m (-1)^{(s+1)(n+1)} \omega_V \circ_s \theta_V, \label{eq:l2V} \\
l_2(\theta,\omega)(f_{n+m-2})
&= - \sum_{s=1}^{n-1} (-1)^{(s+1)(m+1)} \theta_g \circ_s \omega_U
- (-1)^{n+m} \sum_{s=1}^{m-1} (-1)^{(s+1)(n+1)} \omega_g \circ_s \theta_U \notag \\
& \relphantom{} + (-1)^n \sum_{i=1}^n (-1)^{(i-1)m} \theta_V \circ_i \omega_g
+ \sum_{j=1}^m (-1)^{jn} \omega_V \circ_j \theta_g \notag \\
& \relphantom{} + (-1)^{nm + n + m} \theta_g \smallsmile \omega_g + (-1)^{nm} \omega_g \smallsmile \theta_g.  \label{eq:l2g}
\end{align}
\end{subequations}
\end{theorem}

In the statement of the above Theorem, the notations are as in \eqref{eq:circ} and \eqref{eq:thetaiso}, except that
\begin{equation}
\label{eq:thetacircomega}
\begin{split}
\theta_V \circ_i \omega_g
&= \theta_V(g^{\otimes i-1} \otimes \omega_g \otimes g^{\otimes n-i}), \\
\omega_V \circ_j \theta_g
&= \omega_V(g^{\otimes j-1} \otimes \theta_g \otimes g^{\otimes m-j}), \\
\theta_g \smallsmile \omega_g &= \mu_V(\theta_g \otimes \omega_g),
\end{split}
\end{equation}
and similarly for $\omega_g \smallsmile \theta_g$,
$\theta_g \circ_s \omega_U$, and $\omega_g \circ_s \theta_U$.
  In other words, $\theta_V \circ_i \omega_g $ is obtained by plugging $\omega_g$ into the $i$th input of $\theta_V$ and $g$ into the other $(n-1)$ inputs of $\theta_V$.  Likewise, $\theta_g \smallsmile \omega_g$ is simply the usual cup-product of $\theta_g$ and $\omega_g$.

The proof will be given at the end of this section.

Note that Gerstenhaber and Schack did construct a bracket $\lbrack -,- \rbrack$ on their cochain complex $(C^*_{GS}(g;g), d_{GS})$ (see, e.g., the graded commutator bracket of the operation \cite[p.11 (9)]{gs83} or \cite[pp.158-159]{gs88}).  It is straightforward to check that the linear isomorphism \eqref{eq:CASiso} is compatible with $\lbrack -,- \rbrack$ and $l_2$ as well.

\subsection{The operations $l_k$ for $k \geq 3$}

Now consider the cases $k \geq 3$.  Pick elements $\theta_s \in C^{n_s - 1}_{\As_{\texttt{B}\to \texttt{W}}}(T;T)$ $(1 \leq s \leq k)$.  Each $\theta_s$ corresponds, via the identification \eqref{eq:thetaiso}, to the tuple
\[
(\theta_{s,U}, \theta_{s,V}, \theta_{s,g}) = (\theta_s(\mu_{n_s}), \theta_s(\nu_{n_s}), \theta_s(f_{n_s-1})) \in C_{GS}^{n_s}(g;g).
\]
Since $l_k$ has degree $2-k$, the element $l_k(\theta_1, \ldots , \theta_k)$ lies in $C^{t-1}_{\As_{\texttt{B}\to \texttt{W}}}(T;T)$, where
\[
t = 3 - 2k + \sum_{s=1}^k n_s.
\]
Under the identification \eqref{eq:thetaiso}, $l_k(\theta_1, \ldots , \theta_k)$ is uniquely determined by
\[
\left(l_k(\theta_1, \ldots , \theta_k)(\mu_t), l_k(\theta_1, \ldots , \theta_k)(\nu_t), l_k(\theta_1, \ldots , \theta_k)(f_{t-1})\right) \in C^t_{GS}(g;g).
\]

Now we extend the first $\circ_i$ operation in \eqref{eq:thetacircomega} as follows.  Fix $s \in \lbrace 1, \ldots , k \rbrace$.  Let
\[
\mathbf{a} = (a_1, \ldots , \widehat{a_s}, \ldots , a_k)
\]
be a $(k-1)$-tuple of distinct points in the set $\lbrace 1, \ldots , n_s \rbrace$.  Then we define
\begin{equation}
\label{eq:circa'}
\theta_{s,V} \circ_\mathbf{a} (\theta_{1,g}, \ldots , \widehat{\theta_{s,g}}, \ldots , \theta_{k,g}) \in \Hom(U^{\otimes t-1}, V)
\end{equation}
to be the element obtained by plugging $\theta_{j,g}$ ($1 \leq j \leq k$, $j \not= s$) into the $a_j$th input of $\theta_{s,V}$ and $g$ into the other $(n_s - (k-1))$ inputs of $\theta_{s,V}$.  Also define the sign
\[
(-1)^{\mathbf{a}} = (-1)^{\sum_{1 \leq i < j \leq n_s} r_i(r_j + 1)},
\]
where
\[
r_a =
\begin{cases}
\vert \theta_j \vert = n_j - 1 & \text{if } a = a_j \in \lbrace a_1, \ldots , \widehat{a_s}, \ldots, a_k \rbrace, \\
1 & \text{otherwise}.
\end{cases}
\]

\begin{theorem}
\label{thm:lktheta}
For $k \geq 3$ and notations as above, we have
\begin{subequations}
\label{eq:lktheta}
\begin{align}
l_k(\theta_1, \ldots , \theta_k)(\mu_t) &= 0, \label{eq:lkmu}\\
l_k(\theta_1, \ldots , \theta_k)(\nu_t) &= 0, \, \text{and} \label{eq:lknu}\\
l_k(\theta_1, \ldots , \theta_k)(f_{t-1})
&= -(-1)^{\nu(\theta_1, \ldots , \theta_k)} \sum_{s=1}^k \sum_{\mathbf{a}}\, (-1)^{\mathbf{a}} \, \theta_{s,V} \circ_\mathbf{a} (\theta_{1,g}, \ldots , \widehat{\theta_{s,g}}, \ldots , \theta_{k,g}). \label{eq:lkft}
\end{align}
\end{subequations}
Here $\nu(\theta_1, \ldots , \theta_k)$ is defined in \eqref{eq:nu} and, for each $s$, $\mathbf{a} = (a_1, \ldots , \widehat{a_s}, \ldots , a_k)$ runs through all the $(k-1)$-tuples of distinct points in the set $\lbrace 1, \ldots , n_s \rbrace$.
\end{theorem}

\begin{corollary}
\label{cor:lktheta}
Suppose that $k \geq 3$ and that $\theta_s \in C^{n_s - 1}_{\As_{\texttt{B}\to \texttt{W}}}(T;T)$ $(1 \leq s \leq k)$. If
\[
n_s < k-1 \quad \text{for} \quad 1 \leq s \leq k,
\]
then
\[
l_k(\theta_1, \ldots , \theta_k) = 0.
\]
In other words, for each $q \geq 0$ and any $k \geq q + 3$, the operation
\[
l_k \colon \left( C^{\leq q}_{\As_{\texttt{B}\to \texttt{W}}}(T;T)\right)^{\otimes k} \to C^*_{\As_{\texttt{B}\to \texttt{W}}}(T;T)
\]
is trivial.
\end{corollary}

\begin{proof}[Proof of Theorem \ref{thm:l2}]
To prove \eqref{eq:l2U}, first note that
\[
\partial(\mu_{n+m-1}) = \sum_{\substack{i+j\,=\,n+m \\ i,j\,\geq\,2}} \sum_{s=0}^{n+m-1-j} (-1)^{i+s(j+1)} \mu_i \circ_{s+1} \mu_j.
\]
Since the $E$-decorated $2$-colored directed $(1,n+m-1)$-graph $\mu_i \circ_{s+1} \mu_j$ has two vertices, we have
\[
\begin{split}
& l_2(\theta,\omega)(\mu_{n+m-1}) \\
&= (-1)^{\vert \theta \vert} \sum_{\substack{i+j\,=\,n+m \\ i,j\,\geq\,2}} \sum_{s=0}^{n+m-1-j} (-1)^{i+s(j+1)} \left\lbrace\theta(\mu_i) \circ_{s+1} \omega(\mu_j) + \omega(\mu_i) \circ_{s+1} \theta(\mu_j)\right\rbrace \\
&= (-1)^{n-1}\left(\sum_{s=0}^{n-1} (-1)^{n+s(m+1)} \theta_U \circ_{s+1} \omega_U
+ \sum_{s=0}^{m-1} (-1)^{m+s(n+1)} \omega_U \circ_{s+1} \theta_U\right).
\end{split}
\]
This is exactly \eqref{eq:l2U} after a shift of the summation indexes.

Since $\partial(\nu_{n+m-1})$ has the same defining formula as $\partial(\mu_{n+m-1})$ (with $\nu_l$ replacing $\mu_l$ everywhere), the reasoning in the previous paragraph also applies to $l_2(\theta,\omega)(\nu_{n+m-1})$ to establish \eqref{eq:l2V}.

To prove \eqref{eq:l2g}, first note that
\begin{equation}
\label{eq:dfn+m-2}
\begin{split}
\partial(f_{n+m-2})
&= -\sum_{l=2}^{n+m-2} \sum_{r_1 + \cdots + r_l \,=\, n+m-2} (-1)^{\sum_{1 \leq i < j \leq l} r_i(r_j + 1)} \nu_l(f_{r_1} \otimes \cdots \otimes f_{r_l}) \\
& \relphantom{} - \sum_{\substack{i+j\,=\, n+m-1 \\ i \geq 1,\, j \geq 2}} \sum_{s=0}^{n+m-2-j} (-1)^{i+s(j+1)} f_i \circ_{s+1} \mu_j.
\end{split}
\end{equation}
An argument essentially identical to the first paragraph of this proof can be applied to the terms $f_i \circ_{s+1} \mu_j$.  This gives rise to the sums
\begin{equation}
\label{eq:l2f1}
- \sum_{s=1}^{n-1} (-1)^{(s+1)(m+1)} \theta_g \circ_s \omega_U
- (-1)^{n+m} \sum_{s=1}^{m-1} (-1)^{(s+1)(n+1)} \omega_g \circ_s \theta_U
\end{equation}
in $l_2(\theta,\omega)(f_{n+m-2})$.

In \eqref{eq:dfn+m-2}, the $E$-decorated $2$-colored directed $(1,n+m-2)$-graph $\Gamma = \nu_l(f_{r_1} \otimes \cdots \otimes f_{r_l})$ has $l+1$ vertices, say, $v_{top}$, $v^1_{bot}, \ldots , v^l_{bot}$, with decorations $\nu_l$, $f_{r_1}, \ldots , f_{r_l}$, respectively.  In this graph $\Gamma$, the only pairs of distinct vertices are $(v_{top}, v^*_{bot})$, $(v^*_{bot},v_{top})$, and $(v^i_{bot}, v^j_{bot})$ $(i \not= j)$.  The corresponding elements in $l_2(\theta,\omega)(f_{n+m-2})$ (without the signs) are:
\begin{enumerate}
\item $\theta(\nu_l)(\beta(f_{r_1}) \otimes \cdots \otimes \omega(f_{r_i}) \otimes \cdots \beta(f_{r_l}))$ $(1 \leq i \leq l)$, which is $0$ unless $l = n$, $r_i = m-1$, and all the other $r_* = 1$;
\item $\omega(\nu_l)(\beta(f_{r_1}) \otimes \cdots \otimes \theta(f_{r_j}) \otimes \cdots \otimes \beta(f_{r_l}))$ $(1 \leq j \leq l)$, which is $0$ unless $l = m$, $r_j = n-1$, and all the other $r_* = 1$;
\item $\beta(\nu_l)(\beta(f_{r_1}) \otimes \cdots \otimes \theta(f_{r_i}) \otimes \cdots \otimes \omega(f_{r_j}) \otimes \cdots \otimes \beta(f_{r_l}))$, which is $0$ unless $l = 2$ and $(r_1,r_2) = (n-1,m-1)$;
\item $\beta(\nu_l)(\beta(f_{r_1}) \otimes \cdots \otimes \omega(f_{r_i}) \otimes \cdots \otimes \theta(f_{r_j}) \otimes \cdots \otimes \beta(f_{r_l}))$, which is $0$ unless $l = 2$ and $(r_1,r_2) = (m-1,n-1)$.
\end{enumerate}
Taking all the signs into account, we obtain the following sums in $l_2(\theta,\omega)(f_{n+m-2})$:
\begin{equation}
\label{eq:lkf'}
\begin{split}
& -(-1)^{\vert \theta \vert} \sum_{i=1}^n (-1)^{(i-1)(m-1+1)} \theta_V(g^{\otimes i-1} \otimes \omega_g \otimes g^{\otimes n-i}) \\
& -(-1)^{\vert \theta \vert} \sum_{j=1}^m (-1)^{(j-1)(n-1+1)} \omega_V(g^{\otimes j-1} \otimes \theta_g \otimes g^{\otimes m-j}) \\
& -(-1)^{\vert \theta \vert}\left\lbrace (-1)^{(n-1)(m-1+1)} \mu_V(\theta_g \otimes \omega_g) + (-1)^{(m-1)(n-1+1)} \mu_V(\omega_g \otimes \theta_g)\right\rbrace.
\end{split}
\end{equation}
The required result \eqref{eq:l2g} is now obtained by combining \eqref{eq:l2f1} and \eqref{eq:lkf'}. This finishes the proof of Theorem \ref{thm:l2}.
\end{proof}

\begin{proof}[Proof of Theorem \ref{thm:lktheta}]
The computation of $l_k(\theta_1, \ldots , \theta_k)(\mu_t)$ involves choosing $k \geq 3$ distinct vertices in the graphs $\mu_i \circ_{s+1} \mu_j$, each of which has only two vertices.  It follows that
\[
l_k(\theta_1, \ldots , \theta_k)(\mu_t) = 0,
\]
which is \eqref{eq:lkmu}.  The same argument establishes \eqref{eq:lknu}.  Moreover, the same reasoning also shows that the terms $f_i \circ_{s+1} \mu_j$ in $\partial(f_{t-1})$ cannot contribute non-trivially to $l_k(\theta_1, \ldots , \theta_k)(f_{t-1})$.

The remaining statement \eqref{eq:lkft} is now proved by an argument very similar to the last paragraph in the proof of Theorem \ref{thm:l2}.  There is one major difference: In order for the term $\nu_l(f_{r_1} \otimes \cdots \otimes f_{r_l})$ in $\partial(f_{t-1})$ to contribute non-trivially to $l_k(\theta_1, \ldots , \theta_k)(f_{t-1})$, the vertex $v_{top}$ (with decoration $\nu_l$) must be chosen as one of the $k$ distinct vertices because $k \geq 3$ and $\beta(\nu_l) = 0$ for $l \geq 3$.  It follows that each non-trivial term in $l_k(\theta_1, \ldots , \theta_k)(f_{t-1})$ has the form \eqref{eq:circa'}, except for the sign, which is
\[
-(-1)^{\nu(\theta_1, \ldots , \theta_k)} (-1)^{\mathbf{a}}.
\]
The desired condition \eqref{eq:lkft} now follows.
\end{proof}

\section{Deformation complex of a Lie algebra morphism}
\label{sec:liealgmor}

In this section and section \ref{sec:higherbracketlie}, we give a
second illustration of the $L_\infty$-deformation theory of colored
PROP algebras (Section \ref{sec:Linfty}) in the case of Lie algebra
morphisms. The parallelism of the analysis in the associative and
Lie cases shows the unifying character of this approach.

Let $g \colon U \to V$ be a Lie algebra morphism, and set $T = U
\oplus V$ as a $2$-colored graded module.  Let $\Lie_{B \to W}$
denote the $2$-colored operad encoding Lie algebra morphisms. The
purposes of this section are (i)
to express the differential $\delta_{\Lie_{B \to W}}$ in $C^*_{\Lie_{B \to W}}(T;T)$ \eqref{eq:CP} in terms of the Chevalley-Eilenberg
differential (Theorem \ref{thm:deltathetaL}), and
(ii) to observe that the cochain complex $(C^*_{\Lie_{B \to W}}(T;T), \delta_{\Lie_{B \to W}})$ is isomorphic to the S-cohomology cochain complex
$(\Lambda^*(U,V),\Delta^*)$ of the morphism $g$ \cite{fregier,gs05}
(Corollary \ref{cor:deltathetalie}).  This isomorphism allows us to
transfer the $L_\infty$-structure on $(C^*_{\Lie_{B \to W}}(T;T),
\delta_{\Lie_{B \to W}})$ to the S-cohomology cochain complex
$(\Lambda^*(U,V),\Delta^*)$ (Corollary \ref{cor:LinfSC}).

\subsection{Background}

The question of deformation of morphisms of Lie algebras was treated
for the first time by  Nijenhuis and Richardson in \cite{nr3}. The
approach chosen was not the classical method of Gerstenhaber \cite
{ger64}, but the use of the formalization of deformation theory in
terms of graded algebras on the space of cochains developed by
Nijenhuis and Richardson in \cite{nr2}. The starting point was then
the graded Lie algebra on cochains and the differential which were
guessed, the deformation theory being only a corollary. As
drawbacks, the algebras were not allowed to be deformed and the
notion of equivalent deformations was not natural. In order to cure
these two problems, the first author reexamined this problem from
the classical point of view of Gerstenhaber and introduced in
\cite{fregier} the S-cohomology, concluding his work by addressing
the question of the description of a structure for its deformation
complex. Later, in \cite{gs05}, Gerstenhaber, Giaquinto and Schack
showed that this construction is completely parallel to the one
given in
 \cite{gs83} which leads to Diagram cohomology of associative algebras, and hence gave the diagrammatic description of S-cohomology.

\subsection{The S-Cohomology complex $(\Lambda^n(U,V),\Delta^n)$}

\label{subsec:SC}

Here we recall the S-cochain complex $(\Lambda^*(U,V),\Delta^*)$
\cite{fregier}. We modify slightly the notations from \cite{fregier}
to be coherent with the present notations.

Fix a morphism $g \colon U \to V$ of Lie algebras.  We also consider
$V$ as a left $U$-module via $g$.  Then
\begin{equation}
\label{eq:CnSC}  \Lambda^n(U,V) \buildrel \text{def} \over=
\Hom(U^{\wedge {n}}, U) \oplus \Hom(V^{\wedge {n}}, V) \oplus
\Hom(U^{\wedge n-1}, V)\notag
\end{equation}
for $n \geq 1$.
We will also denote a typical element in $\Lambda^n(U,V)$ by
$(x_U,x_V,x_g)$ with $x_U \in \Hom(U^{\wedge {n}}, U)$, $x_V \in
\Hom(V^{\wedge {n}}, V)$, and $x_g \in \Hom(U^{\wedge {n-1}}, V)$.
Its differential is defined as
\begin{equation}
\label{eq:dSC} \Delta^n (x_U, x_V, x_g) \buildrel \text{def} \over=
(bx_U, bx_V, (-1)^{n-1} gx_U -(-1)^{n-1} x_Vg^{\otimes n} + bx_g),
\end{equation}
where $b$ denotes the Chevalley-Eilenberg differential in
$\Hom(U^{\wedge
*}, U)$, $\Hom(V^{\wedge *}, V)$, or $\Hom(U^{\wedge *},V)$.

\subsection{The minimal model of $\Lie_{B \to W}$}

 \label{subsec:minimalLie}

Here we construct the minimal model of the $2$-colored operad $\Lie_{B
\to W}$ that encodes Lie algebra morphisms. This definition is very
similar to the one in the associative category, except for
 the definition of the differential which differs slightly. Moreover one has to be careful
 with respect to the symmetry which is a new feature of the Lie
 category.

The $2$-colored operad $\Lie_{B\to W}$ can be represented as

\[
\Lie_{B\to W} = \frac{\Lie_B \ast \Lie_W \ast \sfF(f)}{(f\mu = \nu
f^{\otimes 2})},
\]
where $\mu$ and $\nu$ denote the generators in $\Lie_B(1,2)$ and
$\Lie_W(1,2)$, respectively, which encode the multiplications in the
domain and the target.

Let $E$ be the $2$-colored $\Sigma$-bimodule with the following skew
symmetric generators:
\begin{equation}
\label{eq:Egenerators}
\begin{split}
\mu_n & \colon B^{\otimes n} \to B \text{ of degree } n-2 \text{ and biarity } (1,n) ~(n \geq 2) \\
\nu_n &\colon W^{\otimes n} \to W \text{ of degree } n-2 \text{ and biarity } (1,n) ~(n \geq 2), \text{ and}\\
f_n & \colon B^{\otimes n} \to W \text{ of degree } n-1 \text{ and
biarity } (1,n) ~(n \geq 1).\notag
\end{split}
\end{equation}
Then the minimal model for $\Lie_{B\to W}$ is
\begin{equation}
\label{eq:minAs} (\sfF(E), \partial) \xrightarrow{\alpha} \Lie_{B\to
W},\notag
\end{equation}
where
\begin{equation}
\label{eq:alphamu} \alpha(\mu_n) = \begin{cases} \mu & \text{if } n
= 2, \\ 0 & \text{otherwise} \end{cases},\, \alpha(\nu_n) =
\begin{cases} \nu & \text{if } n = 2, \\ 0 & \text{otherwise}
\end{cases},\notag
\end{equation}
and
\begin{equation}
\label{eq:alphaf} \alpha(f_n) = \begin{cases} f & \text{if } n = 1,
\\ 0 & \text{otherwise} \end{cases}.\notag
\end{equation}
The differential $\partial$ is given by:
\begin{subequations}
\allowdisplaybreaks
\begin{align}
\partial(\mu_n) & = \sum_{\substack{i+j\,=\,n+1 \\ i,j \,\geq\, 2}} (-1)^{j(i-1)} \sum_{\sigma \in S_{j,i-1}} sgn (\sigma) \mu_i \circ (\mu_j\otimes Id^{\otimes^{i-1}})\circ\sigma, \label{eq:dmul} \\
\partial(\nu_n) & = \sum_{\substack{i+j\,=\,n+1 \\ i,j \,\geq\, 2}}  (-1)^{j(i-1)} \sum_{\sigma \in S_{j,i-1}} sgn (\sigma)\nu_i \circ (\nu_j \otimes Id^{\otimes^{i-1}})\circ\sigma, \label{eq:dnul}\\
\partial(f_n) & = \sum_{l=2}^n \sum_{\substack{r_1 + \cdots + r_l = n \\ r_1\,\leq\, \cdots \,\leq\,r_l}} (-1)^{\frac{l(l-1)}{2}+\sum_{i=1}^{l-1} r_i(l-i)} \sum_{\sigma \in S^{<}_{r_1,\dots,r_l}} sgn (\sigma) \nu_l(f_{r_1} \otimes \cdots \otimes f_{r_l})\circ\sigma \notag \\
 & \relphantom{} - \sum_{\substack{i+j\,=\,n+1 \\ i\,\geq\, 1,j \,\geq\, 2}}  (-1)^{j(i-1)} \sum_{\sigma \in S_{j,i-1}} sgn (\sigma)f_i \circ (\mu_j \otimes Id^{\otimes^{i-1}})\circ\sigma,  \label{eq:dfl}
\end{align}
\end{subequations}
where $S_{j,i-1}$ denotes the set of $j,i-1$ unshuffles and by
$S^{<}_{r_1, \dots,r_l}$ the set of $r_1, \dots,r_l$-unshuffles
satisfying $\sigma (r_1+\dots+r_{i-1}+1)<\sigma (r_1+\dots+r_{i}+1)$
if $r_i=r_{i+1}$. It is also assumed in this notation that $r_i\leq
r_{i+1}$. We refer to \cite{fregier2} for the proof that it is a
minimal model.

One may alternatively write the above formulas with the
summations running over the entire symmetric groups, with coefficients
involving factorials. This would reflect the convention in
describing the morphism of $L_\infty$-algebras used for instance
in~\cite{kajiura-stasheff:JMP06}. The graded anti-symmetry of the
structure operations allows one to bring these formulas into
the above `reduced' form.

\subsection{The cochain complex $(C^*_{\Lie_{B \to W}}(T;T), \delta_{\Lie_{B \to W}})$}

Suppose that $T = U \oplus V$ as a $2$-colored graded module and
that $g \colon U \to V$ is a morphism of Lie algebras represented by
the morphism $\rho \colon \Lie_{B \to W} \to \End_T$ of $2$-colored
operads.  Then the canonical isomorphism \eqref{eq:caniso} says in
this case,
\begin{equation}
\label{eq:canisoL} C^*_{\Lie_{B \to W}}(T;T) =~
\uparrow\Der(\sfF(E), \End_T)^{-*} \cong~
\uparrow\Hom_\Sigma^{\lbrace B,W \rbrace}(E, \End_T)^{-*}.\notag
\end{equation}
Under this isomorphism, an element $\theta \in C^n_{\Lie_{B \to
W}}(T;T)$  is determined by the tuple
\begin{equation}
\label{eq:thetaisoL} \left(\theta_U, \theta_V, \theta_g\right)
\buildrel \text{def} \over= \left(\theta(\mu_{n+1}),
\theta(\nu_{n+1}), \theta(f_n)\right) \in \Lambda^{n+1}(U,V).
\end{equation}
This establishes a linear isomorphism
\begin{align*}
C^n_{\Lie_{B \to W}}(T;T) & \cong  \Lambda^{n+1}(U,V), \\
\theta & \leftrightarrow  \left(\theta_U, \theta_V, \theta_g\right).
\end{align*}

Denote by $\delta$ the differential on the graded module
$\Lambda^*(U,V)$ induced by $\delta_{\Lie_{B \to W}}$. The
identification \eqref{eq:thetaisoL} provides an isomorphism
\begin{equation}
\label{eq:CSisoL} (C^*_{\Lie_{B \to W}}, \delta_{\Lie_{B \to W}})
\xrightarrow{\cong} (\Lambda^{*+1}(U,V), \delta)
\end{equation}
of cochain complexes.

\begin{theorem}
\label{thm:deltathetaL} For $\theta \in C^{n-1}_{\Lie_{B \to
W}}(T;T)$, we have
\begin{equation}
\label{eq:deltathetaL} \delta\left(\theta_U, \theta_V,
\theta_g\right) = \left(b\theta_U, b \theta_V, -b\theta_g
+\theta_Vg^{\otimes n}-g\theta_U\right),
\end{equation}
in which $b$ denotes the appropriate Chevalley-Eilenberg
differential.
\end{theorem}

One can then compare this complex (\ref{eq:deltathetaL}) with the
S-cohomology \eqref{eq:dSC}.

\begin{corollary}
\label{cor:deltathetalie} There is a cochain complex isomorphism
\begin{equation}
\label{eq:pistarlie} \pi = (\pi^*, \pi^*, \tilde\pi^{*}) \colon
(\Lambda^*(U,V), \delta) \xrightarrow{\cong} (\Lambda^*(U,V), \Delta
)\notag
\end{equation} given by
\[
 \begin{cases} \pi^n & = \Id, \\
\tilde\pi^{n} & = (-1)^{n-1}\Id.
\end{cases}
\]
Combined with \eqref{eq:CSisoL}, we obtain an isomorphism
\begin{equation}
\label{eq:CASisoL} (C^*_{\Lie_{B \to W}}(T;T), \delta_{\Lie_{B \to
W}}) \xrightarrow{\cong} (\Lambda^*(U,V), \Delta )
\end{equation}
of cochain complexes.
\end{corollary}

Since $(C^*_{\Lie_{B \to W}}(T;T), \delta_{\Lie_{B \to W}} = l_1,
l_2, l_3, \ldots)$ is an $L_\infty$-algebra (Theorem
\ref{thm:Linfinity}), we can use the cochain complex isomorphism
\eqref{eq:CASisoL} to transfer the higher brackets $l_k$ $(k \geq
2)$ to $(\Lambda^*(U,V), \Delta )$.

\begin{corollary}
\label{cor:LinfSC} There is an $L_\infty$-algebra structure $(\Delta
 = l_1, l_2, l_3, \ldots)$ on $\Lambda^*(U,V)$ capturing deformations
 of the Lie algebra morphism $g$.
\end{corollary}

\begin{proof}[Proof of Theorem \ref{thm:deltathetaL}]
Since $\delta= l_1$ in the $L_\infty$-algebra and since the degree
of $l_1$ is $+1$, we have
\begin{equation}
\label{eq:deltal1L} \delta\left(\theta_U, \theta_V, \theta_g\right)
= (l_1(\theta)(\mu_{n+1}), l_1(\theta)(\nu_{n+1}),
l_1(\theta)(f_n))\notag
\end{equation}
by the identification \eqref{eq:thetaisoL}.  Therefore, to prove
\eqref{eq:deltathetaL}, it suffices to show:
\begin{equation}
\label{eq:deltathetal1L}
\begin{split}
l_1(\theta)(\mu_{n+1}) &= b\theta_U, \\
l_1(\theta)(\nu_{n+1}) &= b \theta_V, \,\text{and} \\
l_1(\theta)(f_n) &= -b\theta_g +\theta_Vg^{\otimes n}-g\theta_U.
\end{split}
\end{equation}

From the description \eqref{eq:lkf} of the operation $l_k$, the
computation of $l_1(\theta)(\mu_{n+1})$ starts with
$\partial(\mu_{n+1})$ \eqref{eq:dmul}. As an $E$-decorated
$2$-colored directed $(1,n+1)$-graph, the term $\mu_i \circ_{s+1}
\mu_j$ in $\partial(\mu_{n+1})$ has two vertices, whose decorations
are $\mu_i$ and $\mu_j$.  Therefore, the expression \eqref{eq:lkf},
when applied to the current situation and using notation
(\ref{eq:circ}), gives
\begin{equation}
\label{eq:l1thetaL} l_1(\theta)(\mu_{n+1}) =
\sum_{\substack{i+j\,=\,n+2 \\ i,j \,\geq\,2}} (-1)^{j(i-1)}
\sum_{\sigma \in S_{j,i-1}} sgn (\sigma)\left\lbrace\theta(\mu_i)
\circ_{1} \beta(\mu_j) + \beta(\mu_i) \circ_{1}
\theta(\mu_j)\right\rbrace\circ \sigma.
\end{equation}

Note that, since $\theta \in C^{n-1}_{\Lie_{B \to W}}(T;T)$,
\[
\theta(\mu_i) = \begin{cases} 0 & \text{if }i \not=n, \\ \theta_U &
\text{if } i=n,\end{cases} \quad \text{and} \quad \beta(\mu_j) =
\rho(\alpha(\mu_j)) = \begin{cases} 0 & \text{if } j \not= 2, \\
\mu_U & \text{if } j = 2,\end{cases}
\]
where $\mu_U \colon U^{\wedge 2} \to U$ is the multiplication on
$U$. It follows that \eqref{eq:l1thetaL} reduces to
\begin{equation}
\label{eq':l1thetaL}
\begin{split}
l_1(\theta)(\mu_{n+1}) &=  \underset{(a)}{\underbrace{(-1)^{2(n-1)}
\sum_{\sigma \in S_{2,n-1}} sgn (\sigma)(\theta_U \circ_{1}
\mu_U)\circ\sigma}}
+  \underset{(b)}{\underbrace{(-1)^{n(2-1)} \sum_{\sigma \in S_{n,1}} sgn(\sigma)(\mu_U \circ_{1} \theta_U)\circ \sigma}}. \notag\\
\end{split}
\end{equation}


In particular, applied to elements of $U,$ (a) and (b) give:

\begin{equation}
\begin{split}
(a)(x_1, \dots, x_{n+1})  & =(-1)^{s+t-1}\sum_{1\leq s<t\leq
n+1}\theta_U (\mu_U(x_s,x_t),x_1, \dots, \hat{x_s}, \dots,
\hat{x_t}, \dots, x_{n+1}),\\
 (b)(x_1, \dots, x_{n+1})& =(-1)^{n(2-1)}\sum_{1\leq
s\leq n+1}(-1)^{n+1-s}\mu_U (\theta_U(x_1, \dots, \hat{x_s}, \dots,
x_{n+1}),x_s)\\
 & =(-1)^{-s}\sum_{1\leq
s\leq n+1}\mu_U (x_s,\theta_U(x_1, \dots, \hat{x_s}, \dots,
x_{n+1})).\notag
\end{split}
\end{equation}
Therefore, by the definition of the Chevalley-Eilenberg
differential, we have
\[
l_1(\theta)(\mu_{n+1})=  b\theta_U,
\]
which is the first condition in \eqref{eq:deltathetal1L}.


The previous paragraph applies verbatim to $l_1(\theta)(\nu_{n+1})$
(with $\nu_l$ replacing $\mu_l$ everywhere), since the definition of
$\partial(\nu_*)$ \eqref{eq:dnul} admits the same formula as that of
$\partial(\mu_*)$.  Therefore, it remains to show the last condition
in \eqref{eq:deltathetal1L}.

In $\partial(f_n)$, the term $\nu_l(f_{r_1} \otimes \cdots \otimes
f_{r_l})$ (respectively, $f_i \circ_{1} \mu_j$) is an $E$-decorated
$2$-colored directed $(1,n)$-graph with $l+1$ (respectively, $2$)
vertices.  Since
\[
\beta(f_j) = \rho(\alpha(f_j)) = \begin{cases} g \colon U \to V &
\text{if } j = 1, \\ 0 & \text{otherwise},\end{cases}
\]
the same kind of analysis as above gives
\begin{equation}
\label{eq:lthetaf}
\begin{split}
l_1(\theta)(f_n) &= \underset {(s_1)}{\underbrace{\sum_{\sigma \in S^{<}_{1,\dots,1}}
sgn(\sigma)\theta_V(g^{\otimes n})\circ\sigma}}
+\underset{(s_2)}{\underbrace{\sum_{\sigma
\in S_{1,n-1}}sgn(\sigma)\mu_V(g
\otimes \theta_g)\circ \sigma } }
\\
 \relphantom{}
 & \hskip -.7em - (-1)^{2(n-1)}\underset{(s_3)}{\underbrace{\sum_{\sigma \in S_{2,n-2}}  sgn(\sigma)\theta_g ( \mu_U \otimes Id^{\otimes^{n-2}})\circ \sigma}}
 - (-1)^{n(1-1)}\underset{(s_4)}{\underbrace{\sum_{\sigma \in S_{n,0}}  sgn(\sigma)g \theta_U \circ \sigma}}.\\
\end{split}
\end{equation}

We now apply the middle summands of (\ref{eq:lthetaf}) on elements
and rearrange them in order to recognize the Chevalley-Eilenberg
differential.  We have
\begin{equation}
\begin{split}
(s_2)(x_1\otimes \dots \otimes x_{n})&  = \sum_{1\leq s\leq
n}(-1)^{s-1} \mu_{V}(g\otimes \theta_{g})(x_{s}\otimes x_1\otimes
\dots \otimes \hat{x_s}\otimes \dots\otimes x_{n} ),\notag
\end{split}
\end{equation}
so \begin{equation}
\begin{split}((s_2)+(s_3))(x_1\otimes \dots \otimes x_{n}) &=+ \sum_{1\leq s\leq
n}(-1)^{s-1} \mu_{V}(g(x_{s}),\theta_{g}(x_1,\dots,
\hat{x_s},\dots,x_{n}))\\
&-\sum_{1\leq s<t\leq
n}(-1)^{s-1+t-2}\theta_{g}(\mu_{U}(x_s,x_t),x_1,\dots,\hat{x_s},
\dots,\hat{x_t},\dots,x_n)\\
& = -b\theta_g (x_1\otimes \dots \otimes x_{n}).\notag\\
\end{split}
\end{equation}

Considering the fact that both $S_{n,0}$ and $S^{<}_{1,\dots ,1 }$
consist of a single element, the trivial permutation, one finally
gets
$$l_1(\theta)(f_n)=-b\theta_g +\theta_Vg^{\otimes n}-g\theta_U,$$
which ends the proof of Theorem \ref{thm:deltathetaL}.
\end{proof}

\section{The higher brackets in $C^*_{\Lie_{B \to W}}(T;T)$}
\label{sec:higherbracketlie}

We keep the same setting and notations as in the previous section.
The purpose of this section is to make explicit the
$L_\infty$-operations $l_k$ on $C^*_{\Lie_{B \to W}}(T;T)$ for $k
\geq 2$.  The cases $k = 2$ (Theorem \ref{thm:l2L}) and $k \geq 3$
(Theorem \ref{thm:lkthetaL}) are treated separately. As an immediate
consequence of our explicit formula for $l_k$ $(k \geq 3)$, we
observe that, when applied to the tensor powers of $C^{\leq
q}_{\Lie_{B\to W}}(T;T)$ for some fixed $q \geq 0$, only
$\delta_{\Lie_{B\to W}}(T;T) = l_1, l_2, \ldots , l_{q+2}$ can be
non-trivial (Corollary \ref{cor:lkthetaL}).

\subsection{The operation $l_2$}

First we deal with the case $k = 2$.  Pick elements $\theta \in
C^{n}_{\Lie_{B \to W}}(T;T)$ and $\omega \in C^{m}_{\Lie_{B \to
W}}(T;T)$.  Under the identification \eqref{eq:thetaisoL}, $\theta$
and $\omega$ correspond to
\[
(\theta_U,\theta_V,\theta_g) \in \Lambda^n(U,V) \quad \text{and}
\quad (\omega_U,\omega_V,\omega_g) \in \Lambda^m(U,V),
\]
respectively.

Since $l_2$ has degree $0$, the element $l_2(\theta,\omega)$ lies in
$C^{n+m}_{\Lie_{B \to W}}(T;T)$.  Under the identification
\eqref{eq:thetaisoL}, $l_2(\theta,\omega)$  is
determined by
\[
(l_2(\theta,\omega)(\mu_{n+m+1}), l_2(\theta,\omega)(\nu_{n+m+1}),
l_2(\theta,\omega)(f_{n+m})) \in \Lambda^{n+m}(U,V).
\]

\begin{theorem}
\label{thm:l2L} With the notations above, we have
\begin{subequations}
\allowdisplaybreaks
\begin{align}
l_2(\theta,\omega)(\mu_{n+m+1}) &= (-1)^{mn} (\sum_{\sigma \in
S_{m+1,n}} sgn (\sigma)
 \theta_U \circ_1 \omega_U\circ\sigma\notag\\
 & \relphantom{}+(-1)^{m+n}\sum_{\sigma
\in S_{n+1,m}} sgn (\sigma)\omega_U \circ_1 \theta_U\circ\sigma),\label{eq:l2UL}\\
l_2(\theta,\omega)(\nu_{n+m+1}) &= (-1)^{mn} (\sum_{\sigma \in
S_{m+1,n}} sgn (\sigma)
 \theta_V \circ_1 \omega_V\circ\sigma\notag\\
 & \relphantom{}+(-1)^{m+n}\sum_{\sigma
\in S_{n+1,m}} sgn (\sigma)\omega_V \circ_1 \theta_V\circ\sigma),\label{eq:l2VL}\\
l_2(\theta,\omega)(f_{n+m}) &=(-1)^{m(n-1)}(\sum_{\sigma \in
S_{m+1,n-1}} sgn (\sigma)
 \theta_g \circ_1 \omega_U\circ\sigma\notag +\sum_{\sigma \in S_{n+1,m-1}} sgn
(\sigma)\omega_g \circ_1\theta_U\circ\sigma)\\
& \relphantom{}(-1)^{n} (\sum_{\sigma \in S^{<}_{1,\dots,1,m}}
 sgn (\sigma)\theta_V(g^{\otimes n} \otimes \omega_g)\circ \sigma + \sum_{\sigma \in S^{<}_{1,\dots,1,n}}
 sgn (\sigma)\omega_V(g^{\otimes m} \otimes \theta_g )\circ \sigma )\notag \\
& \relphantom{} - ( \sum_{\sigma \in S^{<}_{n,m}}
 sgn (\sigma)\mu_V(\theta_g \otimes \omega_g)\circ \sigma +
(-1)^{n+m}\sum_{\sigma \in S^{<}_{m,n}}
 sgn (\sigma) \mu_V(\omega_g \otimes
\theta_g)\circ \sigma). \label{eq:l2gL}
\end{align}
\end{subequations}
\end{theorem}

In the above Theorem, we use the notation of
\eqref{eq:circ} and of $S^{<}$ as defined after (\ref{eq:dfl}). One
should remark that depending on whether $m<n$ or $n<m$, the first or the
second summand in the last bracketed expression above is zero.
The proof of the Theorem will be given at the end of this section.


\subsection{The operations $l_k$ for $k \geq 3$}

Now consider the cases $k \geq 3$.  Pick elements $\theta_s \in
C^{n_s}_{\Lie_{B\to W}}(T;T)$ $(1 \leq s \leq k)$.  Each $\theta_s$
corresponds, via the identification \eqref{eq:thetaisoL}, to the
tuple
\[
(\theta_{s,U}, \theta_{s,V}, \theta_{s,g}) = (\theta_s(\mu_{n_s+1}),
\theta_s(\nu_{n_s+1}), \theta_s(f_{n_s})) \in \bigwedge^{n_s}(U;V).
\]
Since $l_k$ has degree $2-k$, the element $l_k(\theta_1, \ldots ,
\theta_k)$ lies in $C^{t}_{\Lie_{B\to W}}(T;T)$, where
\[
t =- k+2 + \sum_{s=1}^k n_s.
\]
Under the identification \eqref{eq:thetaisoL}, $l_k(\theta_1, \ldots
, \theta_k)$ is determined by
\[
\left(l_k(\theta_1, \ldots , \theta_k)(\mu_{t+1}), l_k(\theta_1,
\ldots , \theta_k)(\nu_{t+1}), l_k(\theta_1, \ldots ,
\theta_k)(f_{t})\right) \in \bigwedge^{t}(U;V).
\]

Now we extend the $\circ_i$ operation as follows. Fix $s \in \lbrace
1, \ldots , k \rbrace$.  Let
\begin{equation}
\label{eq:boldaL} \mathbf{a}' = (a_1, \ldots , \widehat{a_s}, \ldots
, a_k)\notag
\end{equation}
be a $(k-1)$-tuple of distinct points in the set $\lbrace 1, \ldots
, n_s+1 \rbrace$.  Then we define
\begin{equation}
\label{eq:circa} \theta_{s,V} \circ_\mathbf{a'} (\theta_{1,g},
\ldots , \widehat{\theta_{s,g}}, \ldots , \theta_{k,g}) \in
\Hom(U^{\otimes t}, V)\notag
\end{equation}
to be the element obtained by plugging $\theta_{j,g}$ ($1 \leq j
\leq k$, $j \not= s$) into the $a_j$th input of $\theta_{s,V}$ and
$g$ into the other $(n_s +2 - k)$ inputs of $\theta_{s,V}$.  Also
define the coefficient
\begin{equation}
\label{eq:boldasign}
(-1)^{\mathbf{a}'}=(-1)^{\frac{(n_s+1)(n_s)}{2}+\sum_{ i = n_s+1}
r_i(n_s+1-i)},\notag
\end{equation}
where
\[
r_a =
\begin{cases}
\vert \theta_j \vert = n_j & \text{if } a = a_j \in \lbrace a_1, \ldots , \widehat{a_s}, \ldots, a_k \rbrace, \\
1 & \text{otherwise}.
\end{cases}
\]
Let us remark that the set $\{r_1,\dots,r_{n_{s}+1}\}$ satisfies
$r_1+\dots+r_{n_{s}+1}=t$. One says that $\mathbf{a}'$ is \emph{admissible}
if this set also satisfies $r_1\leq \dots \leq r_{n_{s}+1}$. One
denotes by $A$ the set of admissible $\mathbf{a}'$.

\begin{theorem}
\label{thm:lkthetaL} For $k \geq 3$ and notations as above, we have
\begin{subequations}
\label{eq:lkthetaL}
\begin{align}
& l_k(\theta_1, \ldots , \theta_k)(\mu_{t+1}) = 0,\notag \notag\\
& l_k(\theta_1, \ldots , \theta_k)(\nu_{t+1}) = 0, \, \text{and} \notag\\
& l_k(\theta_1, \ldots , \theta_k)(f_{t}) = (-1)^{\nu(\theta_1,
\ldots , \theta_k)} \sum_{s=1}^k \sum_{\mathbf{a'}\in A}\ \sum_{\sigma
\in S^{<}_{r_1,\dots,r_{n_s+1}}}\hskip -1.5em sgn (\sigma) (-1)^{\mathbf{a'}}
\theta_{s,V} \circ_\mathbf{a'} (\theta_{1,g}, \ldots ,
\widehat{\theta_{s,g}}, \ldots , \theta_{k,g})\!\circ\!\sigma, \notag
\end{align}
\end{subequations}
with $\nu(\theta_1, \ldots , \theta_k)$ defined as in \eqref{eq:nu},
and $S^{<}$ defined after \eqref{eq:dfl}.
\end{theorem}

\begin{corollary}
\label{cor:lkthetaL} Suppose that $k \geq 3$ and that $\theta_s \in
C^{n_s}_{\Lie_{B\to W}}(T;T)$ $(1 \leq s \leq k)$. If
\[
n_s < k-1 \quad \text{for} \quad 1 \leq s \leq k,
\]
then
\[
l_k(\theta_1, \ldots , \theta_k) = 0.
\]
In other words, for each $q \geq 0$ and any $k \geq q + 3$, the
operation
\[
l_k \colon \left( C^{\leq q}_{\Lie_{B\to W}}(T;T)\right)^{\otimes k}
\to C^*_{\Lie_{B\to W}}(T;T)
\]
is trivial.
\end{corollary}

\begin{proof}[Proof of Theorem \ref{thm:l2L}]
To prove \eqref{eq:l2UL}, first note that
\begin{equation}
\label{eq:dmun+m+1}
\partial(\mu_{n+m+1}) = \sum_{\substack{i+j\,=\,n+m+2 \\ i,j \,\geq\, 2}} (-1)^{j(i-1)}
\sum_{\sigma \in S_{j,i-1}} sgn (\sigma) \mu_i \circ (\mu_j\otimes
Id^{i-1})\circ\sigma.\notag
\end{equation}
Since the $E$-decorated $2$-colored directed $(1,n+m+1)$-graph
$\mu_i \circ (\mu_j\otimes Id^{i-1})$ has two vertices, we have
\begin{equation}
\allowdisplaybreaks
\begin{split}
& l_2(\theta,\omega)(\mu_{n+m+1}) \\
&= (-1)^{\vert \theta \vert} \hskip -.2em   \sum_{\substack{i+j\,=\,n+m+2 \\ i,j
\,\geq\, 2}}\hskip -.5em
(-1)^{j(i-1)} \sum_{\sigma \in S_{j,i-1}} \hskip -.5em  sgn (\sigma)
\left\lbrace \theta(\mu_i) \circ (\omega(\mu_j)\otimes
g^{\otimes^{i-1}})+\omega(\mu_i) \circ (\theta(\mu_j)\otimes
g^{\otimes^{i-1}})\right\rbrace\circ\sigma\\
 &=
(-1)^{mn} \hskip -.2em   \sum_{\sigma \in S_{m+1,n}} sgn (\sigma)
 \theta_U \circ (\omega_U\otimes
g^{\otimes^{n}})\circ\sigma+(-1)^{(n+1)m+n}\sum_{\sigma \in
S_{n+1,m}} sgn (\sigma)\omega_U \circ (\theta_U\otimes
g^{\otimes^{m}})\circ\sigma.\notag
\end{split}
\end{equation}

Since $\partial(\nu_{n+m+1})$ has the same defining formula as
$\partial(\mu_{n+m+1})$ (with $\nu_l$ replacing $\mu_l$ everywhere),
the reasoning in the previous paragraph also applies to
$l_2(\theta,\omega)(\nu_{n+m+1})$ to establish \eqref{eq:l2VL}.

To prove \eqref{eq:l2gL}, first note that
\begin{equation}
\label{eq:dfn+m}
\begin{split}
\partial(f_{n+m}) & = \sum_{l=2}^{n+m} \sum_{\substack{r_1 + \cdots + r_l = n+m \\ r_1\,\leq\, \cdots \,\leq\,r_l}} (-1)^{\frac{l(l-1)}{2}+\sum_{i=1}^{l-1} r_i(l-i)} \sum_{\sigma \in S^{<}_{r_1,\dots,r_l}} sgn (\sigma) \nu_l(f_{r_1} \otimes \cdots \otimes f_{r_l})\circ\sigma \notag \\
 & \relphantom{} - \sum_{\substack{i+j\,=\,n+m+1 \\ i\,\geq\, 1,j \,\geq\, 2}}  (-1)^{j(i-1)} \sum_{\sigma \in S_{j,i-1}} sgn (\sigma)f_i \circ (\mu_j \otimes Id^{\otimes^{i-1}})\circ\sigma,
\end{split}
\end{equation}
An argument essentially identical to the first paragraph of this
proof can be applied to the terms $f_i \circ (\mu_j \otimes
Id^{i-1})$.  This gives rise to the sums
\begin{equation}
\label{eq:l2f1L} (-1)^{m(n-1)}(\sum_{\sigma \in S_{m+1,n-1}} sgn
(\sigma)
 \theta_g \circ (\omega_U\otimes
g^{\otimes^{n-1}})\circ\sigma+\sum_{\sigma \in S_{n+1,m-1}} sgn
(\sigma)\omega_g \circ (\theta_U\otimes
g^{\otimes^{m-1}})\circ\sigma)
\end{equation}
in $l_2(\theta,\omega)(f_{n+m})$.

In \eqref{eq:dfn+m}, the $E$-decorated $2$-colored directed
$(1,n+m)$-graph $\Gamma = \nu_l(f_{r_1} \otimes \cdots \otimes
f_{r_l})$ has $l+1$ vertices, say, $v_{top}$, $v^1_{bot}, \ldots ,
v^l_{bot}$, with decorations $\nu_l$, $f_{r_1}, \ldots , f_{r_l}$,
respectively.  In this graph $\Gamma$, the only pairs of distinct
vertices are $(v_{top}, v^*_{bot})$, $(v^*_{bot},v_{top})$, and
$(v^i_{bot}, v^j_{bot})$ $(i \not= j)$.  The corresponding elements
in $l_2(\theta,\omega)(f_{n+m})$ (without the signs) are:
\begin{enumerate}
\item $\theta(\nu_l)(\beta(f_{r_1}) \otimes \cdots \otimes \omega(f_{r_i}) \otimes \cdots \beta(f_{r_l}))$ $(1 \leq i \leq l)$, which is $0$ unless $l = n+1$, $r_{n+1} = m$, and all the other $r_* = 1$;
\item $\omega(\nu_l)(\beta(f_{r_1}) \otimes \cdots \otimes \theta(f_{r_j}) \otimes \cdots \otimes \beta(f_{r_l}))$ $(1 \leq j \leq l)$, which is $0$ unless $l = m+1$, $r_{m+1} = n$, and all the other $r_* = 1$;
\item $\beta(\nu_l)(\beta(f_{r_1}) \otimes \cdots \otimes \theta(f_{r_i}) \otimes \cdots \otimes \omega(f_{r_j}) \otimes \cdots \otimes \beta(f_{r_l}))$, which is $0$ unless $l = 2$ and $(r_1,r_2) = (n,m)$;
\item $\beta(\nu_l)(\beta(f_{r_1}) \otimes \cdots \otimes \omega(f_{r_i}) \otimes \cdots \otimes \theta(f_{r_j}) \otimes \cdots \otimes \beta(f_{r_l}))$, which is $0$ unless $l = 2$ and $(r_1,r_2) = (m,n)$.
\end{enumerate}
Taking all the signs into account, we also obtain the following sums
in $l_2(\theta,\omega)(f_{n+m})$:
\begin{equation}
\label{eq:lkf'L}
\begin{split}
& (-1)^{\vert \theta \vert} \sum_{\sigma \in S^{<}_{1,\dots,1,m}}
 sgn (\sigma)\theta_V(g^{\otimes n} \otimes \omega_g)\circ \sigma \\
& (-1)^{\vert \theta \vert} \sum_{\sigma \in S^{<}_{1,\dots,1,n}}
 sgn (\sigma)\omega_V(g^{\otimes m} \otimes \theta_g )\circ \sigma \\
& (-1)^{\vert \theta \vert}\left\lbrace (-1)^{1+n}\sum_{\sigma \in
S^{<}_{n,m}}\hskip -.25em
 sgn (\sigma)\mu_V(\theta_g \otimes \omega_g)\circ \sigma +
(-1)^{1+m}\sum_{\sigma \in S^{<}_{m,n}}\hskip -.25em
 sgn (\sigma) \mu_V(\omega_g \otimes
\theta_g)\circ \sigma\right\rbrace.
\end{split}
\end{equation}
The required result \eqref{eq:l2gL} is now obtained by combining
\eqref{eq:l2f1L} and \eqref{eq:lkf'L}. This finishes the proof of
Theorem \ref{thm:l2L}.
\end{proof}

\begin{proof}[Proof of Theorem \ref{thm:lkthetaL}]
This proof is identical to the proof of Theorem \ref{thm:lktheta} if
one shifts the indices, replacing $t$ by $t+1$, $s+1$ by $1$, and
$(-1)^{\mathbf{a}}$ by $-(-1)^{\mathbf{a}'}$
\end{proof}


\section{Deformations of diagrams}
\label{MAGDA}

Recall that a {\em diagram\/} in a category $\calC$ is a functor $\calF :
\calD \to \calC$ from a small category $\calD$ to $\calC$; the
category $\calD$ is called the {\em shape\/} of the diagram
$\calF$. Diagrams of shape $\calD$
can equivalently be described  as algebras
over an $\Ob(\calD)$-colored operad $\bfD$ which
has  only elements of arity $1$ (one input, one output) and
\[
\bfD\binom{d}{c} := {\rm Mor\/}_\calD (c,d),\  \mbox{ for $c,d \in
  \Ob(\calD)$.}
\]
The operadic composition in $\bfD$ equals the categorial composition
of $\calD$.  It is clear that $\calD$-diagrams in $\calC$ are
precisely $\Ob(\calD)$-colored $\bfD$-algebras in $\calC$.

The  $\Ob(\calD)$-colored
operad $\bfD$ as defined above lives in the category of sets. Since
we will be primarily interested in diagrams in
the category of $\bk$-vectors spaces, we may as well consider
the $\bk$-linear operad generated by
$\bfD$ or assume from the very beginning that $\bfD$ is given by
\[
\bfD\binom{d}{c} :=
{\rm Span\/}_\bk\left({\rm Mor\/}_\calD (c,d)\right),\  \mbox{ for $c,d \in
  \Ob(\calD)$,}
\]
where ${\rm Span\/}_\bk(-)$ denotes the $\bk$-linear span. We will
call colored operads of the above form {\em diagram operads\/}.

\begin{example}
\label{Iso}
In this example we describe the diagram operad $\Iso$ associated to the
category $\mathcal{I}\mathit{so}$ consisting of two objects and two
mutually inverse maps between these objects.
Let $f: \texttt{B} \to \texttt{W}$, $g: \texttt{W} \to \texttt{B}$ be two degree-zero generators. Then
\[
\Iso := \frac{\sfF(f,g)}{(fg=1_\texttt{W},\ gf=1_\texttt{B})},
\]
where $\sfF(f,g)$ denotes the free
$\{\texttt{B},\texttt{W}\}$-colored operad on the set
$\{f,g\}$ and $(fg=1_\texttt{W},\ gf=1_\texttt{B})$ the operadic ideal
generated by $fg-1_\texttt{W}$ and $gf-1_\texttt{B}$.

Algebras over $\Iso$
consist of two mutually inverse
degree zero chain maps $F : U \to V$ and $G : V \to U$.
In other words, $\Iso$-algebras are diagrams
\begin{equation}
\label{napise_mi_Magda_jeste?}
{
\unitlength=.45pt
\begin{picture}(180.00,40.00)(80.00,30.00)
\put(20.00,10.00){\vector(-3,1){1.00}}
\put(80.00,52.00){$F$}
\put(80.00,-22.00){$G$}
\put(160.00,30.00){\vector(3,-1){1.00}}
\bezier{100}(20.00,10.00)(90.00,-10.00)(160.00,10.00)
\bezier{100}(20.00,30.00)(90.00,50.00)(160.00,30.00)
\put(180.00,20.00){\makebox(0.00,0.00){$V$}}
\put(0.00,20.00){\makebox(0.00,0.00){$U$}}
\put(187.00,20.00){\makebox(0.00,0.00)[l]{,\ $FG=1$ \ and \ $GF =1$.}}
\end{picture}}
\end{equation}
\vskip 1.6em
\end{example}

A typical diagram operad $\bfD$ does not admit a minimal model.  For
instance, a hypothetical minimal model of the operad $\Iso$ from
Example~\ref{Iso} shall have generators $f_0$, $g_0$ for $f$ and $g$,
but also a generator, say $f_1$, whose boundary kills the difference
$f_0g_0-1_\texttt{W}$, i.e.~satisfying
\[
f_0g_0-1_\texttt{W} = \pa f_1.
\]
The ``constant" $1_\texttt{W}$ however defies any thinkable notion of
minimality.

This phenomenon is related to the fact that a typical diagram operad
$\bfD$, such as $\Iso$, is not augmented, by which we mean that it
does not admit an operad morphism $\bfD \to {\mathbf i}$ to the
terminal $\Ob(\calD)$-colored operad ${\mathbf i}$. In the next
example we will see that sometimes there still exists a cofibrant
resolution whose size is that of a minimal model.

\begin{example}
\label{Magda}
A small cofibrant resolution of $\Iso$ was described
in~\cite[Theorem~9]{markl:ip}.
It is a graded colored differential operad
\[
\Riso := (\sfF(f_0,f_1,\ldots; g_0,g_1,\ldots), d),
\]
with generators of two types,
\[
\def\arraystretch{1.2}
\begin{array}{rl}
\mbox{(i)}&\mbox{%
\hskip -2mm
generators $\seq{f_n}{n \geq 0}$, $\deg(f_n)=n$,}
\left\{
\begin{array}{l}
\mbox{$f_n : \texttt{B} \to \texttt{W}$ if $n$ is even,}
\\
\mbox{$f_n : \texttt{B} \to \texttt{B}$ if $n$ is odd,}
\end{array}
\right.
\\
\mbox{(ii)}&\mbox{%
\hskip -2mm
generators
$\seq{g_n}{n \geq 0}$, $\deg(g_n)=n$,}
\left\{
\begin{array}{l}
\mbox{$g_n : \texttt{W} \to \texttt{B}$ if $n$ is
even,}
\\
\mbox{$g_n : \texttt{W} \to \texttt{W}$ if $n$ is odd.}
\end{array}
\right.
\end{array}
\]
The differential $\pa$ is given by
\[
\def\arraystretch{1.4}
\begin{array}{ll}
\pa f_0\widedef 0,
&
\pa g_0 \widedef 0,
\\
\pa f_1 \widedef g_0f_0 - 1,
&
\pa g_1 \widedef f_0g_0 - 1
\end{array}
\]
and, on remaining generators, by the formula
\begin{eqnarray*}
\nonumber
\pa f_{2m}   &:=& \sum_{0 \leq i < m}(f_{2i}f_{2(m-i)-1}-
g_{2(m-i)-1}f_{2i}),\ m \geq 0,
\\
\label{pejsek_a_kocicka}
\pa f_{2m+1} &:=& \sum_{0 \leq j \leq m} g_{2j}f_{2(m-j)} -
               \sum_{0 \leq j < m} f_{2j+1}f_{2(m-j)-1},\ m \geq 1,
\\
\nonumber
\pa g_{2m}   &:=& \sum_{0 \leq i < m}(g_{2i}g_{2(m-i)-1}-
f_{2(m-i)-1}g_{2i}),\ m \geq 0,
\\
\nonumber
\pa g_{2m+1} &:=& \sum_{0 \leq j \leq m} f_{2j}g_{2(m-j)} -
               \sum_{0 \leq j < m} g_{2j+1}g_{2(m-j)-1},\ m \geq 1.
\end{eqnarray*}

The above resolution is ``minimal' in the following sense.
Consider a one-parametric family $\Iso_\varepsilon$
of $\{\texttt{B},\texttt{W}\}$-colored operads defined by
\[
\Iso_\varepsilon
:= \frac{\sfF(f,g)}{(fg=\varepsilon \cdot 1_\texttt{W},
gf= \varepsilon\cdot 1_\texttt{B})},
\]
where $\varepsilon$ is a formal parameter.
The operad $\Iso_0$ clearly describes couples $(F,G)$ of
maps $F :U \to V$ and $G : V \to U$ such that $FG = 0$ and $GF = 0$,
while $\Iso_\varepsilon$ is, at a generic $\varepsilon$, isomorphic to
the operad $\Iso$. In other words, $\Iso$ is a deformation of
$\Iso_0$.
It turns out that $\Iso_0$ is an augmented colored operad that admits
a minimal model whose generators are the same as the generators of
$\Riso$; see~\cite[Theorem~10]{markl:ip}.
\end{example}

An obvious generalization of the
machinery developed in the previous sections applies verbatim
to resolutions of diagram operads.
One typically gets an $L_\infty$-algebra with a nontrivial
`curvature' $l_0$; see~\cite[Section~5]{markl07b} for the terminology
and definitions.
The corresponding Maurer-Cartan equation
then involves the $l_0$-term.

\begin{example}
\label{ex:LinftyIso}
Let us describe the $L_\infty$-deformation complex
$(\CIso*,l_0,l_1,l_2,\ldots)$ for a diagram $T$ as
in~(\ref{napise_mi_Magda_jeste?}). As we already observed, this
deformation complex has, as a consequence of the presence of $1$
in the formulas for $\pa f_1$ and $\pa g_1$ in $\Riso$, a nontrivial
$l_0$. On the other hand, since the differential $\pa$ on the
generators of $\Riso$ does not have higher than quadratic terms, all
$l_k$'s are trivial, for $k\geq 3$.

Formula~(\ref{eq:CP}) applied to
the resolution $\Riso$ from Example~\ref{Magda} gives
the underlying cochain complex
\[
\CIso n =
\begin{cases}
\Hom(U,V) \oplus \Hom(V,U) & \text{for } n \geq 1 \text{ odd, and} \\
\Hom(U,U) \oplus \Hom(V,V) & \text{for } n \geq 1 \text{ even.}
\end{cases}
\]
Formula~(\ref{eq:lkf}) makes sense also for $k=0$ and describes
$l_0 \in \CIso2$ as the direct sum of the identity maps $\Id_U \oplus
\Id_V \in \Hom(U,U) \oplus \Hom(V,V)$.

Likewise, one obtains the following formulas for the operation
\[
l_1: \CIso * \to \CIso{*+1}.
\]
If $\alpha \oplus \beta \in \CIso n$, $n \geq 1$ odd, then
\[
l_1(\alpha \oplus \beta) = (G \alpha + \beta F) \oplus (\alpha G +
F\beta)
\in \CIso {n+1}.
\]
For $\gamma \oplus \delta \in \CIso n$, $n \geq 1$ even, we have
\[
l_1(\gamma \oplus \delta) = (F\gamma - \delta F) \oplus (G\delta
-\gamma G) \in \CIso {n+1}.
\]

The bracket
\[
l_2: \CIso m \otimes \CIso n \to \CIso {m+n}
\]
is given as follows. For $\alpha'
\oplus \beta' \in \CIso m$, $\alpha'' \oplus \beta'' \in \CIso n$, $m,n$
odd, we have
\[
l_2(\alpha' \oplus \beta',\alpha'' \oplus \beta'') =
(\beta'\alpha'' + \beta''\alpha') \oplus (\alpha' \beta '' + \alpha''
\beta')
\in \CIso {m+n}.
\]
For $\alpha
\oplus \beta \in \CIso m$, $\gamma \oplus \delta \in \CIso n$, $m$
odd, $n$ even, we have
\[
l_2(\alpha \oplus \beta,\gamma \oplus \delta) =
- l_2(\gamma \oplus \delta,\alpha \oplus \beta) =
(\alpha\gamma - \delta\alpha) \oplus (\beta\delta - \gamma\beta)
\in \CIso {m+n}.
\]
Finally, for $\gamma' \oplus \delta' \in \CIso m$,
$\gamma'' \oplus \delta'' \in \CIso n$,
$m,n$ even, we have
\[
l_2(\gamma' \oplus \delta',\gamma'' \oplus \delta'') =
-l_2(\gamma'' \oplus \delta'',\gamma' \oplus \delta')
=
(\gamma'\gamma'' - \gamma''\gamma') \oplus (\delta'\delta'' - \delta''\delta')
\in \CIso {m+n}.
\]
The higher $l_k$'s, $k \geq 3$, are trivial.

Observe that, for each $w \in \CIso *$, $l_2(l_0,w) = 0$; therefore,
by~\cite[Section~5]{markl07b},
$l_1^2 =0$. In other words, $l_1$ is a differential
and the standard analysis of deformation theory
applies. For instance,
there exists the canonical element $\chi := F \oplus G  \in
\CIso1$ such that
\[
l_1(w) = l_2(\chi,w),\ w \in \CIso *.
\]

As expected~\cite{markl07b,van},
$\Iso$-algebras are solutions of the Maurer-Cartan equation
in the $L_\infty$-complex  of the trivial $\bfD$-algebra $T_{\mbox{\o}}$
with $F=0$, $G=0$. Indeed, if $\kappa = \Phi \oplus \Psi
\in \CIso1$, then the Maurer-Cartan equation
\[
-l_0  + \frac 12 l_2(\kappa,\kappa) = 0
\]
expands into
\[
-(\Id_U \oplus \Id_V) + \frac 12(2\Psi \Phi \oplus 2\Phi \Psi) = 0,
\]
which says precisely that $\Phi$ and $\Psi$ are mutually inverse
isomorphisms.
\end{example}



\end{document}